%% file: main.tex
\documentclass[11pt]{article}
\usepackage{amsfonts,amssymb,amsthm,eucal,amsmath,verbatim}
\usepackage{setspace}

\usepackage[english]{babel}
\usepackage{tikz}
\usepackage[letterpaper,top=1in,bottom=1in,left=1in,right=1in,marginparwidth=1.75cm]{geometry}
\usepackage{subfigure}

\usepackage{mdframed}
\usepackage{graphicx}
\usepackage{mathtools}
\usepackage[colorlinks=true, allcolors=blue]{hyperref}
\usepackage{url} 
\usepackage{subfiles} 
\usepackage{amsfonts}
\usepackage{amsthm}
\usepackage{subcaption}
\usepackage[section]{placeins}
\usepackage{bbm} 
\usepackage{enumitem}
\usepackage{lmodern}
\usepackage{textcomp} 

\usepackage{mathtools}

\newtheorem{theorem}{Theorem}[section]
\newtheorem{corollary}[theorem]{Corollary}
\newtheorem{lemma}[theorem]{Lemma}
\newtheorem{definition}[theorem]{Definition}

\newtheorem{proposition}[theorem]{Proposition}

\usepackage[normalem]{ulem}

\numberwithin{equation}{section}
\numberwithin{theorem}{section}

\title{Mesoscopic Rates of Convergence for Complex Wishart Matrices at the Leftmost Spectrum Edge}
\date{}
\author{Mengchun Cai%
\thanks{Department of Mathematics, Applied Mathematics, and
  Statistics, Case Western Reserve University, Cleveland, Ohio,
  U.S.A.; mengchun.cai@case.edu.}}

\begin{document}
\maketitle
\input{1_Introduction}

\input{2_Results}
\input{3_Proof}
\input{Appendix}
\bigskip
\noindent{\bf Acknowledgements. }The results of this paper are part of the author's PhD thesis supervised by Professor Mark W. Meckes and Professor Stanislaw J. Szarek. The author is indebted to them for providing important discussions and valuable feedback. Supported in part by the NSF Grant DMS-$1600124$.

\newpage

\bibliographystyle{plain}
\bibliography{my_references}

\end{document}

%% file: 1_Introduction.tex
\begin{abstract}
\noindent This paper establishes mesoscopic rates of convergence in the $L^1$-Wasserstein distance for eigenvalue determinantal point processes (DPPs) derived from the Laguerre Unitary Ensemble (LUE) to the corresponding limiting point processes as the dimension goes to infinity. Specifically, we prove convergence rates at the leftmost edge of the LUE spectrum, which corresponds to the least eigenvalue. These results are termed mesoscopic because they allow for a direct comparison of point counts between the convergent DPPs and their limits over a range of scales. Our approach relies on controlling the trace class norm of the integral operators defined by the DPP kernels. As an important byproduct, the mesoscopic analysis also yields the convergence rate of the least eigenvalue to the Tracy–Widom distribution $TW_2$ for the complex Wishart ensemble (the same as the LUE when the dimension satisfies certain conditions).
\end{abstract}
\
\section{Introduction}
Random matrix theory (RMT) is an active subfield of mathematics, statistics and physics. It was initiated in statistics and introduced to physics by Wigner and Dyson in the 1950s-1960s. Given its mathematical richness, RMT has become very popular since the end of 1970s through the development of quantum physics, statistical mechanics, probability theory, geometric functional analysis,  theoretical computer science and qualitative finances. Here we cite \cite{forrester_developments_2003} as a well-rounded introduction for the history of RMT.

The literature on RMT is huge and varied, extending into its applications to a wide variety of seemingly unrelated domains, ranging from stability of ecosystems and financial markets to zeros of the Riemann $\zeta$ function (for comprehensive introduction, see \cite{akemann_oxford_2015,anderson2010introduction,mehta_random_2004,pastour_eigenvalue_2011,potters_first_2020}).

Naturally, in RMT, many of the most important questions revolve around the (random) spectrum of different classes of random matrices. The main purpose of this paper is to analyze non-asymptotic (or, perhaps better, ``beyond asymptotic") results on the leftmost spectrum edge of large Wishart matrices. Let us fix the matrix models that we will be studying.
\begin{definition}
  Consider a matrix $X\in M^{n\times m}(\mathbb{C})$ with all i.i.d.\ complex Gaussian entries such that
    \begin{equation*}
        \mathbb{E} X_{j,k}=\mathbb{E}X_{j,k}^2=0,~\mathbb{E}|X_{j,k}|^2=1
    \end{equation*}
    for any $j,~k$. Then the Hermitian matrix $W_{n,m}:=X^*X\in M^{m}(\mathbb{C})$ is called a \textbf{complex Wishart matrix}, or a \textbf{Laguerre Ensemble}.
\end{definition}

The Wishart ensemble, serving as a matrix analog of the $\chi^2$ random variable, first emerged in the late 1920s within multivariate statistics and was later adopted by other fields due to its wide applicability (see \cite{muri-multi-1982} for an introduction). Specifically, we refer to \cite{guhr_random-matrix_1998} for applications in quantum mechanics, \cite{aubrun_alice_2017} in quantum information theory, \cite{vershynin_high-dimensional_2018} in data analysis, \cite{couillet_random_2023} in machine learning and \cite{plerou_random_2002} in finance.

The work of this paper owes a lot to the pioneering work of Tracy and Widom \cite{tracy1994level}, \cite{tracy1996orthogonal} and also to that of Johnstone \cite{johnstone2001distribution}, El Karoui \cite{el2006rate} and Ma \cite{ma_accuracy_2012}.

\subsection{Determinantal Point Process (DPPs)}
The basic idea in this paper is to analyze the eigenvalues of a random matrix with large dimension. The eigenvalues can be regarded as a random point process.
\begin{definition}\label{DPP}
 A point process (PP) $\mathfrak{X}$ in a locally compact Polish space $\Lambda$ is a random discrete subset of $\Lambda$. For any $A\subset\Lambda$,  $\mathcal{N}_\mathfrak{X}(A)$ denotes the number of points of $\mathfrak{X}$ in $A$. The integer-valued random function: $A\rightarrow\mathcal{N}_\mathfrak{X}(A)$
is called the \textbf{counting function} of $\mathfrak{X}$.   
\end{definition}
 In this paper, we mainly focus on point processes in $\Lambda=\mathbb{R}$. Given a Borel set $A\subset\mathbb{R}$ and a random matrix $H$, we use the notation $\mathcal{N}_H(A)$ to represent the number of eigenvalues of $H$ contained in $A$. Moreover, we refer to points in a point process as eigenvalues even when they are not true eigenvalues.

Fix a point process $\mathfrak{X}$ on $\mathbb{R}$. If they exist, the \textbf{correlation functions} (also called joint intensities) for $\mathfrak{X}$ are a sequence of locally integrable functions $\{\rho_k:\mathbb{R}^k\rightarrow\mathbb{R} \}_{k=1}^{\infty}$ satisfying the following condition: for all mutually disjoint measurable subsets $\{D_j\}_{j=1}^k$ of $\mathbb{R}$, 
\begin{equation*}
	\mathbb{E} \left[\prod_{j=1}^{k} \mathcal{N}_{\mathfrak{X}}(D_j)\right] = \int_{\prod_j D_j} \rho_k(x_1,\dots,x_k) dx_1 \cdots dx_k.
\end{equation*}   
Analogous to the classic moment problem from probability, under many circumstances correlation functions uniquely specify a point process (see \cite{soshnikov2000determinantal}, where it is shown that this is the case for DPPs).
\begin{definition}
  Let $\mathcal{X}$ be a PP in $\mathbb{R}$. Suppose that there exists a function $K:\mathbb{R}^2\rightarrow\mathbb{C}$ such that, for almost every $x_j\in\mathbb{R}$,
\begin{equation*}
    \rho_k(x_1,\dots,x_k)=\det\left(K(x_j,x_l)_{j,l=1}^k\right),
\end{equation*}
holds for any $k\in\mathbb{N}$. Then $\mathfrak{X}$ is called a \textbf{determinantal~point~process}, and $K$ is called the \textbf{kernel} of $\mathfrak{X}$.  
\end{definition}
As a result, for a DPP, the kernel determines the correlation functions, which themselves determine the PP. Hence, studying the corresponding kernel leads to results for a DPP.

Besides the random matrix ensembles that will be introduced later, let's briefly review two common DPPs we will encounter when considering the leftmost edge of the LUE
\begin{definition}\label{def4 classic DPPs}
An \textbf{Airy~point~process} $\mathfrak{X}_{Ai}$ is a DPP given by the kernel:
\begin{equation}\label{airy kernel}
    K_{Ai}(x,y)=\left\{\begin{array}{rr}
      \frac{Ai(x)Ai'(y)-Ai(y)Ai'(x)}{x-y},~~x\ne y\\\\
      \left(Ai'(x)\right)^2-xAi^2(x),~~x=y
    \end{array}\right.
\end{equation}
for $(x,y)\in\mathbb{R}^2$, where
\begin{equation*}
    Ai(x)=\frac{1}{\pi}\int_0^\infty\cos\left(\frac{t^2}{3}+xt\right)dt
\end{equation*}
is the \textbf{Airy~function}.

Similarly, a \textbf{Bessel~point~process} $\mathfrak{X}_{Bes,a}$ with parameter $a\ge0$ is a DPP with the kernel:
\begin{equation}\label{bessel kernel}
    K_{Bes,a}(x,y)=\left\{\begin{array}{cc}
    \frac{\sqrt y J'_a(\sqrt y)J_a(\sqrt x)-\sqrt x J'_a(\sqrt x)J_a(\sqrt y)}{2(x-y)},~~x\ne y\\\\
    \frac{x-a^2}{4x}\left(J_a(\sqrt x)\right)^2+\frac{1}{4}\left(J'_a(\sqrt x)\right)^2,~~x=y
    \end{array}\right.
\end{equation}
for $x,~y\ge0$, where
\begin{equation*}
    J_a(x)=\sum_{m=0}^\infty\frac{(-1)^m}{\Gamma(m+1)\Gamma(m+a+1)}\left(\frac{x}{2}\right)^{2m+a}
\end{equation*}
is the \textbf{Bessel~function} with parameter $a$.
\end{definition}

\subsection{The LUE and the M--P Law}
The eigenvalue distribution addressed in this paper is also referred to as the following important ensemble in RMT.
\begin{proposition}\label{ei4LUE}\cite[Proposition $3.2.2$]{log-gas}
   Suppose $a\in\mathbb{N}$. Let $X$ be an $(N+a)\times N$ matrix with i.i.d.\ complex standard Gaussian entries, then the joint distribution of the eigenvalues of $L_{a,N}:=X^*X$ has the following probability density function:
    \begin{equation}\label{J4LUE}
        p_{L_{a,N},2}(x_1,x_2,...,x_N)=\frac{1}{C_{L_{a,N},2}}\prod_{1\le j< k\le N}\chi_{\{x_j\ge0\}}|x_j-x_k|^2\sum_{j=1}^N x_j^a\exp(-x_j)
    \end{equation}
    for $C_{L_{a,N},2}$ appropriate normalization constant.
    
    In addition, an $N$-dimensional nonnegative random vector $(x_1,x_2,..,x_N)$ with the joint probability density function as \eqref{J4LUE} is called a \textbf{Laguerre Unitary Ensemble (LUE)} with parameter $a$.
   For convention, we say the Wishart matrix $L_{a,N}$ belongs to the LUE.
\end{proposition}
Although the general definition for the LUE makes sense whenever $a>-1$, we only focus on the case where $a\in\mathbb{N}$ because of the corresponding matrix model in this setting.

One well-known fact is that the LUE with parameter $a\in\mathbb{N}$ forms a DPP with the kernel:
\begin{equation}\label{LUE kernel}
    K_{L_a,N}(x,y)=\frac{-\Gamma(N+1)}{\Gamma(N+a)}\frac{\psi_{L_a,N}(x)\psi_{L_a,N-1}(y)-\psi_{L_a,N}(y)\psi_{L_a,N-1}(x)}{x-y}
\end{equation}
where
\begin{equation*}
\begin{split}
       \psi_{L_a,j}(x):=x^{\frac{a}{2}}e^{-\frac{x}{2}}\chi_{\{x\ge0\}}L_j^a(x)=x^{\frac{a}{2}}e^{-\frac{x}{2}}\chi_{\{x\ge0\}}(-1)^j j!\sum_{k=0}^j\left(\begin{matrix}
        j+a\\j-k
    \end{matrix}\right)\frac{(-1)^k}{k!}x^k
\end{split}
\end{equation*}
is the $j$th \textbf{Laguerre function} with parameter $a$ and $L_j^a(x)$
is the $j$th \textbf{Laguerre polynomial} with parameter $a$ (see \cite[Equations $(5.46)~to~(5.48)$]{log-gas}).

Given a Wishart matrix $W_{m,n}$, the following probability measure is closely related to its \textbf{Empirical Spectrum Measure} (ESM).
\begin{definition}\label{ESM}
    Given an Hermitian matrix $H\in M^n(\mathbb{C})$, the ESM of $H$ is defined as:
    \begin{equation*}
        \mu_H:=\frac{1}{n}\sum_{j=1}^n\delta_{\lambda_j}
    \end{equation*}
    where $\lambda_j$'s are eigenvalues of $H$.
\end{definition}
\begin{definition}\label{M-P law}
    Given two parameters $\lambda>0$ and $\sigma>0$, define
    \begin{equation*}
        g_{\lambda,\sigma}(x):=\frac{1}{2\pi\sigma^2}\frac{\sqrt{(\lambda^+-x)(x-\lambda^-)}}{\lambda x}\chi_{[\lambda^-,\lambda^+]}
    \end{equation*}
    where $\lambda^{\pm}:=\sigma^2(1\pm\sqrt\lambda)^2$. We say a probability measure $\rho_{\lambda,\sigma}$ follows the \textbf{Marchenko--Pastur distribution} with parameter $\lambda$ and $\sigma$ if and only if for any measurable $A\subset\mathbb{R}$,
    \begin{equation*}
        \rho_{\lambda,\sigma}(A)=\left\{\begin{array}{l}
          \chi_{\{0\in A\}}(1-\lambda^{-1})+\mu_{\lambda,\sigma}(A),~\lambda>1\\
          \\
          \mu_{\lambda,\sigma}(A),~\lambda\in(0,1]
        \end{array}\right.
    \end{equation*}
    where $\mu_{\lambda,\sigma}$ has the density $g_{\lambda,\sigma}$ with respect to the Lebesgue measure on $\mathbb{R}$.
\end{definition}
For a Wishart matrix $W_{n,m}$, the following \textbf{Marchenko--Pastur law}, which was first proved by Marchenko and Pastur in \cite{marcenko_distribution_1967}, holds for its ESM.
\begin{proposition}
    \cite[Theorem $3.6\&3.7$]{bai_spectral_2010} Let $W_{n,m}$ be Wishart matrices (real or complex) with $\lim\limits_{n\rightarrow\infty}\frac{m}{n}=\gamma\in(0,\infty)$. Then with probability one, the ESMs of $W'_{n,m}:=\frac{1}{n}W_{n,m}$ converge to $\rho_{\gamma,1}$ in distribution. 
\end{proposition}
For a sequence of Wishart matrices $W_{n,m}=X^*X$ such that $\frac{m}{n}\rightarrow\gamma>1$, we see it must have a bunch of $0$'s, according to its rank, as the eigenvalues when $n$ is large. These $0$ eigenvalues are considered trivial as they provide us with little interesting information but a mass around $0$ as in the M--P distribution. Noticing that $X^*X$ and $XX^*$ have the same nontrivial eigenvalues, we always assume $\gamma\in(0,1]$ in this paper and only draw attention to $L_{a,N}$ as similar discussion applies to $XX^*$ when $\gamma>1$. 

Given an LUE $L_{a,N}$, the density $g_{1,\gamma}$ already identifies $(1-\sqrt\gamma)^2$ and $(1+\sqrt\gamma)^2$ as two spectrum edges. At the right edge---the expected position of the largest eigenvalue---points converge to the Airy point process after proper scaling as the dimension goes to infinity (see \cite{cai2025mesoscopicratesconvergencehermitian}, \cite{el2006rate} and \cite{johnstone2001distribution}). Since there is no bound that absolutely limits the size of the largest eigenvalue (as argued in \cite{bai1993convergence} and \cite{geman_limit_1980}, although it is possible for the largest eigenvalue to exceed its typical location, the probability of this event is very low), the right edge is often referred to as the \textbf{soft edge}.

The behavior at the left edge, which is the main focus of this paper, is limited by the fact that the least eigenvalue of a positive semidefinite matrix must be absolutely bounded below by $0$. Consequently, both the location and nature vary according to the value of limiting ratio $\lim\limits_{N\rightarrow\infty}\frac{N}{N+a}:=\gamma\in(0,1]$. In fact, the authors in \cite{hachem_large_2016} established a comprehensive classification of the edges of the support of the ESM for complex Wishart matrices (LUE). Simply following their framework, we restrict our attention to two regular regimes for the left edge of LUE. First, $\lim\limits_{N\rightarrow\infty }a(N)=a$, or equivalently, $a(N)$ is a fixed integer for $N$ large. In this setting, the left edge of $L_{a,N}$ touches $0$, which is a natural bound for the eigenvalues as $L_{a,N}$ is a positive semidefinite matrix. Under this assumption, $0$ is called the \textbf{regular~hard~edge} of $L_{a,N}$. After proper rescaling, around the regular hard edge, authors in \cite{cai2025mesoscopicratesconvergencehermitian} established the ROC of the LUE towards the Bessel point process with parameter $a$. Second, the ratio $\gamma$ is in $(0,1)$. In this case, the left edge is called a \textbf{regular~soft~edge} as it is bounded away from $0$. Near this edge, the main purpose of this paper is to show that when the dimension gets large, the LUE converges to the \textbf{symmetric~Airy~point~process} after proper scaling.

Given a Polish metric space $(\Omega,d)$, two probability measures $\mu$ and $\nu$ on $\Omega$ and $p\ge1$, the $L^p$-Wasserstein distance $W_p$ between $\mu$ and $\nu$ is defined as:
\begin{equation*}
W_p(\mu,\nu):=\\\inf_{(x,y)\in\pi(\mu,\nu)}\left(\mathbb{E}d^p(x,y)\right)^{\frac{1}{p}}
\end{equation*}
where $\pi(\mu,\nu)$ contains all couplings of $\mu$ and $\nu$. The choice of $W_1$ in this paper is motivated by a powerful lemma that serves as the starting point for our proofs. This metric is widely used in probability theory and in particular, in optimal transport (see \cite[Chapter $6$]{villani2009optimal}). For the limiting behaviors mentioned above, our goal in this paper is to provide a ROC for different ensembles with respect to the $L^1$-Wasserstein distance $W_1$.

The following result in \cite{cai2025mesoscopicratesconvergencehermitian} and \cite{el2006rate} illustrates the Airy-type fluctuation for the LUE around the right edge. We include it here to contrast with our later result for the left edge in Theorem \ref{LUE left soft}.
\begin{theorem}\label{LUE right soft}
   (\cite[Theorem $2.5$]{cai2025mesoscopicratesconvergencehermitian}, right soft edge) Consider an LUE $L_{a,N}$ such that $\lim\limits_{N\rightarrow\infty}\frac{N}{n}=\gamma\in(0,1]$ where $n:=N+a$. Let $I\subset[s,+\infty)$ for $s>-\infty$. Define parameters:
    \begin{equation*}
      \begin{array}{ll}
      \mu^R_{k,m}:=\left(\sqrt{k+\frac{1}{2}}+\sqrt{m+\frac{1}{2}}\right)^2,\\
      \sigma^R_{n,m}:=\left(\sqrt{k+\frac{1}{2}}+\sqrt{m+\frac{1}{2}}\right)\left(\left(k+\frac{1}{2}\right)^{-\frac{1}{2}}+\left(m+\frac{1}{2}\right)^{-\frac{1}{2}}\right)^{\frac{1}{3}},
      \end{array}
    \end{equation*}
\begin{equation*}
    \mu_R=\left(\frac{1}{\sqrt{\sigma_{n-1,N}^{R}}}+\frac{1}{\sqrt{\sigma_{n,N-1}^{R}}}\right)\left(\frac{1}{\mu^R_{n-1,N}\sqrt{\sigma_{n-1,N}^{R}}}+\frac{1}{\mu^R_{n,N-1}\sqrt{\sigma_{n,N-1}^{R}}}\right)^{-1}
\end{equation*}
and
\begin{equation*}
    \sigma_R=\left(\frac{\sqrt{\sigma_{n-1,N}^{R}}}{\mu^R_{n-1,N}}+\frac{\sqrt{\sigma_{n,N-1}^{R}}}{\mu^R_{n,N-1}}\right)\left(\frac{1}{\mu^R_{n-1,N}\sqrt{\sigma_{n-1,N}^{R}}}+\frac{1}{\mu^R_{n,N-1}\sqrt{\sigma_{n,N-1}^{R}}}\right)^{-1}.
\end{equation*}
Let $\mathcal{X}_{L_{a,N}}^{RS}$ denote the right-soft-edge-scaled point process for $L_{a,N}$ with the kernel $K_{L_{a,N}}^{RS}(x,y):=\sigma_{R}K_{L_{a,N}}(t_{R_{N}}(x),t_{R_{N}}(y))$ where $t_{R_N}(x):=\mu_{R}+\sigma_{R}x$ for any $x\in I$. Then on $A\subset I$, for any $s\in\mathbb{R}$, there exists an $n(s,\gamma)\in\mathbb{N}$ such that for any $N>n(s,\gamma)$,
\begin{equation}\label{main L right soft}
    W_1(\mathcal{N}_{L_{a,N}}^{RS},\mathcal{N}_{Ai})\le\frac{g(s)e^{-s}}{N^{\frac{2}{3}}}
\end{equation}
for $g(s)$ some continuous non-increasing function of $s$.
\end{theorem}
Similar to our discussion for the least eigenvalue later, equation \eqref{main L right soft} immediately implies ROCs of order $N^{-\frac{2}{3}}$ for the largest eigenvalues in the LUE to the Tracy--Widom distribution. Moreover, the Liouville--Green method used in \cite{el2006rate} serves as the main tool for our analysis in this paper.

\subsection{Outline of Proof}
For a kernel function $K(x,y)$ on some domain $D^2\subset\mathbb{R}^2$, we can define a corresponding integral operator $\mathbf{K}$ acting on $L^2(D)$ by the formula:
\begin{equation*}
    \mathbf{K}(f)(x)=\int_D K(x,y)f(y)dy
\end{equation*}
for any $f\in L^2(D)$. Recall further that there is a hierarchy of classes of such (compact) operators (again analogous to $L^p$ spaces) based on the rate of decay of the singular values of the operator known as the \textbf{Schatten $p$-Norms}. Given a compact operator $\mathbf{T}$, it belongs to the Schatten $p$-class if
\begin{equation*}
 \lVert \mathbf{T} \rVert_{S_p}:= \left( \sum_{n=1}^{\infty} s_n(\mathbf{T})^p\right)^{\frac{1}{p}} < \infty
\end{equation*}
where $s_n(\mathbf{T})$'s are singular values of $\mathbf{T}$. For convenience, for a given suitable operator $\mathbf{L}$, we simply denote its Schatten $p$-norm as $\|\mathbf{L}\|_p$ in this paper.

As expected, the two most important cases are $p=1$ and $p=2$. The $p=1$ case are known as trace-class operators and the $p=2$ case are called Hilbert--Schmidt operators.

Soshnikov proved in \cite{soshnikov2000determinantal} that the kernel of a self-adjoint, locally trace-class operator defines a determinantal point process if and only if all the eigenvalues of the operator are between $0$ and $1$. Consequently, it is possible to translate the question regarding random point counts to the analysis of the corresponding integral operators.

The following lemmas in \cite{cai2025mesoscopicratesconvergencehermitian} provide us with tangible bounds for $W_1$ distance between random point counts.

\begin{lemma}\label{M and M}\cite[Lemma $1.1$]{cai2025mesoscopicratesconvergencehermitian}
	Consider two DPPs $\mathfrak{X}$ and $\widetilde{\mathfrak{X}}$ with Hermitian kernels $K(x,y)$ and $\widetilde{K}(x,y)$ and associated integral operators $\mathbf{K}$ and $\widetilde{\mathbf{K}}$ respectively. Assume the integral operators are trace class. Then,
	\begin{equation*}
	d_{TV}(\mathcal{N}_\mathcal{X},\mathcal{N}_{\tilde{\mathcal{X}}})\le W_1(\mathcal{N}_\mathcal{X},\mathcal{N}_{\tilde{\mathcal{X}}})
	\leq \lVert \mathbf{K} - \widetilde{\mathbf{K}} \rVert_1
	\end{equation*}
    where $d_{TV}$ is the total variation distance.
\end{lemma}
\begin{lemma}\label{21in}\cite[Lemma $3.1$]{cai2025mesoscopicratesconvergencehermitian}
    Suppose $\mathbf{A}$, $\mathbf{B}$ and $\mathbf{C}$ are Hilbert--Schmidt operators on a Hilbert space. Then the following inequality holds:
    \begin{equation*}
2\|\mathbf{A}\mathbf{B}+\mathbf{B}\mathbf{A}-\mathbf{C}\mathbf{C}\|_1\le\|\mathbf{A}+\mathbf{B}-\sqrt2\mathbf{C}\|_2\|\mathbf{A}+\mathbf{B}+\sqrt2\mathbf{C}\|_2+\|\mathbf{A}-\mathbf{B}\|^2_2.
    \end{equation*}
      Moreover, for any Hilbert--Schmidt operator $\mathbf{K}$ with the integral kernel $K$, it follows that $\|\mathbf{K}\|_2=\|K\|_{L^2}$.
\end{lemma}
Lemma \ref{21in} is useful since the exact computation of the trace-class norm is complicated, whereas estimating the $L^2$ norm of the kernel is more tractable.
Therefore, the work in the proofs boils down to two things: decomposing the operators appropriately so that the inequality in Lemma \ref{21in} can be used, and getting non-asymptotic approximation towards the Airy function (since it serves as the limiting kernel in our setting) to finish the estimates.

The following operator decomposition, first introduced in \cite{tracy1994level} and then extended in \cite{johnstone2001distribution}, plays an important role in our reasoning.
\begin{proposition}\label{LUE fac}
    Let $L_{a,N}$ be an LUE with parameter $a\in\mathbb{N}$ and kernel $K_{L_{a,N}}(x,y)$. Then for $a\ge 2$, the following decomposition holds for the corresponding integral operator $\mathbf{K}_{L_{a,N}}$:
 \begin{equation}\label{LUE int2}
\mathbf{K}_{L_{a,N}}=\mathbf{G}_{a,N}\mathbf{H}_{a,N}+\mathbf{H}_{a,N}\mathbf{G}_{a,N}.
\end{equation}
Thereinto, the operators $\mathbf{G}_{a,N}$ and $\mathbf{H}_{a,N}$ have kernels $\xi_{a,N}(x+y)$ and $\eta_{a,N}(x+y)$ for
\begin{equation}\label{explicit xi and zeta}
    \begin{array}{ll}
   \xi_{a,N}(x):=\frac{(-1)^N}{\sqrt2}\left(N(N+a)\right)^\frac{1}{4}\sqrt{\frac{\Gamma(N+1)}{\Gamma(N+a)}}x^{\frac{a}{2}-1}e^{-\frac{x}{2}}L_{N}^{a-1}(x)\chi_{(0,\infty)}(x),\\
   \\
    \eta_{a,N}(x):=\frac{(-1)^{N-1}}{\sqrt2}\left(N(N+a)\right)^\frac{1}{4}\sqrt{\frac{\Gamma(N)}{\Gamma(N+a+1)}}x^{\frac{a}{2}}e^{-\frac{x}{2}}L_{N-1}^{a+1}(x)\chi_{(0,\infty)}(x)
    \end{array}
\end{equation}
respectively.

Let $\mathbf{K}_{Ai}$ be the integral operator with the Airy kernel defined in \eqref{airy kernel}, then
\begin{equation}\label{Afac}
    \mathbf{K}_{Ai}=\mathbf{Ai}\mathbf{Ai}
\end{equation}
where $\mathbf{Ai}$ has the kernel $Ai(x+y)$.

Moreover, all operators above are Hilbert-Schmidt on $L^2(0,+\infty)$.
\end{proposition}

%% file: 2_Results.tex
\section{Statement of Results}\label{Sec: Results}
The following is the main result proved in this paper.
\begin{theorem}\label{LUE left soft}
   (left soft edge) Consider an LUE $L_{a,N}$ such that $\lim\limits_{N\rightarrow\infty}\frac{N}{n}=\gamma\in(0,1)$ where $n:=N+a$. Let $I\subset[s,+\infty)$ for $s>-\infty$. Define parameters:
    \begin{equation*}
      \begin{array}{ll}
      \mu^L_{k,m}:=\left(\sqrt{k+\frac{1}{2}}-\sqrt{m+\frac{1}{2}}\right)^2,\\
      \sigma^L_{k,m}:=\left(\sqrt{k+\frac{1}{2}}-\sqrt{m+\frac{1}{2}}\right)\left(\left(m+\frac{1}{2}\right)^{-\frac{1}{2}}-\left(k+\frac{1}{2}\right)^{-\frac{1}{2}}\right)^{\frac{1}{3}},
      \end{array}
    \end{equation*}
\begin{equation*}
    \mu_L=\left(\frac{1}{\sqrt{\sigma_{n-1,N}^{L}}}+\frac{1}{\sqrt{\sigma_{n,N-1}^{L}}}\right)\left(\frac{1}{\mu^L_{n-1,N}\sqrt{\sigma_{n-1,N}^L}}+\frac{1}{\mu^L_{n,N-1}\sqrt{\sigma_{n,N-1}^{L}}}\right)^{-1}
\end{equation*}
and
\begin{equation*}
    \sigma_L=\left(\frac{\sqrt{\sigma_{n-1,N}^{L}}}{\mu^L_{n-1,N}}+\frac{\sqrt{\sigma_{n,N-1}^{L}}}{\mu^L_{n,N-1}}\right)\left(\frac{1}{\mu^L_{n-1,N}\sqrt{\sigma_{n-1,N}^{L}}}+\frac{1}{\mu^L_{n,N-1}\sqrt{\sigma_{n,N-1}^{L}}}\right)^{-1}.
\end{equation*}
Let $\mathcal{X}_{L_{a,N}}^{LS}$ denote the left-soft-edge-scaled point process for $L_{a,N}$ with the kernel $K_{L_{a,N}}^{LS}(x,y):=\sigma_{L}K_{L_{a,N}}(t_{L_{N}}(x),t_{L_{N}}(y))$ where $t_{L_N}(x):=\mu_{L}-\sigma_{L}x$ for any $x\in I$. Then on $A\subset I$, there exists $n(s,\gamma)\in\mathbb{N}$ such that for any $N>n(s,\gamma)$,
\begin{equation}\label{main L left soft}
W_1(\mathcal{N}_{L_{a,N}}^{LS},\mathcal{N}_{Ai})\le\frac{g(s,\gamma)e^{-s}}{N^{\frac{2}{3}}}
\end{equation}
where $g(s,\gamma)$ is continuous and non-increasing with respect to $s$ for any fixed $\gamma$.
\end{theorem}

\bigskip\noindent\textbf{ Remark.} The scaling we used in Theorem \ref{LUE left soft} closely mirrors that in Theorem \ref{LUE right soft}, first introduced in \cite{el2006rate}, due to the similar L–G method employed in the large 
$N$ approximation. However, unlike Theorem \ref{LUE right soft}, the function
$g(s,\gamma)$ in Theorem \ref{LUE left soft} depends on the limiting ratio $\gamma$. This difference arises because, in our estimation of the least eigenvalue, we are unable to construct a uniform bound with respect to $\gamma$ in the same way as we can for the largest eigenvalue in Theorem \ref{LUE right soft}. A more detailed explanation will be provided later.\bigskip

Theorem \ref{LUE left soft} also provides us with a convergence rate for the least eigenvalue $\lambda_{\min}$ of an LUE towards the Tracy--Widom distribution as following. 
 \begin{corollary}\label{smallessoft}
    Let $L_{a,N}$ be an LUE with parameters satisfying $\lim\limits_{N\rightarrow\infty}\frac{N}{N+a}=\gamma\in(0,1)$. Suppose $\lambda_{min}^N$ is the least eigenvalue of the $L_{a,N}$ and $TW_2(s)$ is the distribution function of the least eigenvalue of an Airy point process. Let $\tilde{\lambda}^N_{min}:=\frac{\lambda_{min}^N+\mu_L}{\sigma_L}$ be the left-soft-edge-scaled least eigenvalue of $L_{a,N}$. Then for large $N$,
    \begin{equation*}
        |\mathbb{P}\{-\tilde{\lambda}_{min}^N\le s\}-TW_2(s)|\le \frac{g(s,\gamma)e^{-s}}{N^{\frac{2}{3}}}.
    \end{equation*}
    where $g$ is the same function as in Theorem \ref{LUE left soft}.
\end{corollary}
\begin{proof}
    Notice $-\lambda_{min}^N=\mu_L-\sigma_L\tilde{\lambda}_{min}^N$, which matches the explicit scaling in Theorem \ref{LUE left soft}, then \eqref{main L left soft} applies. The same as the previous proof, it follows that
   \begin{equation*}
\begin{split}
     |\mathbb{P}\{-\tilde{\lambda}_{min}^N\le s\}-TW_2(s)|&=|\mathbb{P}\{\mathcal{N}_{L_{a,N}}^{LS}((s,\infty))=0\}-\mathbb{P}\{\mathcal{N}_{Airy}((s,\infty))=0\}|
     \\&\le d_{TV}(\mathcal{N}_{L_{a,N}}^{LS}((s,\infty)),\mathcal{N}_{Airy}((s,\infty)))
     \\&\le W_1(\mathcal{N}_{L_{a,N}}^{LS}((s,\infty)),\mathcal{N}_{Airy}((s,\infty)))\\&\le\frac{g(s,\gamma)e^{-s}}{N^\frac{2}{3}}
\end{split}
\end{equation*}
where $g$ is the same function as in Theorem \ref{LUE left soft}.
\end{proof}
The distribution $TW_2(s)$, with an explicit form given in \cite{tracy1994level}, is the famous \textbf{unitary~Tracy--Widom~distribution}. The result in Corollary \ref{smallessoft} actually covers the ROC, with a better spacial function $g(s,\gamma)$, derived by Ma in \cite[Theorem $2$]{ma_accuracy_2012}. Moreover, the negative sign in this corollary indicates the fact that after proper scaling, the point process derived by the eigenvalues of $L_{a,N}$ around the left soft edge converges to the symmetric Airy point process, which slightly differs from the right soft one. Similar asymptotics can be seen in \cite{ma_accuracy_2012} in the same setting and \cite{DePaul-asyp} for $\gamma=0$.

\subsection{Comparison with Other Results}
Unlike a microscopic analysis, which is confined to a fixed, local scale, our approach permits the scaling parameter $s$ to change, subject to asymptotic conditions that ensure we remain within the mesoscopic regime. This dynamic scaling covers some the rigid setting of microscopic limits, for example, Forrester's discussion for the convergence of the kernel functions in \cite{log-gas}, and provides a powerful tool for probing the spatial structure of the point process across different scales. Indeed, aside from extreme values of the parameter (as the small $s$ for the left hard edge and large $s$ for the left soft edge for the LUE), we cannot expect our results to be optimal in the spatial parameter $s$, as the estimates used in the proof are not sharp. However, for any fixed $s$, we conjecture that our results are optimal with respect to the dimension $N$.

The method we use relies heavily on the auxiliary variable introduced by the Liouville--Green transform. Following the analysis in the proof, it appears unrealistic to expect a rate higher than $N^{\frac{2}{3}}$, as the power $\frac{2}{3}$ is intrinsic in our L--G transform around the least eigenvalue (in fact, it is also necessary around the largest one as in \cite{el2006rate}). As a result, it seems very likely that the order $N^{-\frac{2}{3}}$ is optimal in this case.

As indicated in Corollary \ref{smallessoft}, the mesoscopic results around the left edge of the LUE imply a the ROC for the distribution of its least eigenvalue. However, in contrast to the right edge behavior studied in \cite{el2006rate}, the behavior of the left edge varies a lot according to the limiting value of the ratio $\frac{N}{N+a}$.

%% file: 3_Proof.tex
\section{Proofs}\label{Proof}
Similar to the method used in \cite{el2006rate} and \cite{johnstone2001distribution},
in this section we aim to derive the convergence rate by comparing the large $N$ expansion of $\xi$ and $\eta$ (perhaps with different forms they are defined in \eqref{LUE int2}) with the Airy function $Ai(x)$. Following the same philosophy of analyzing $\xi$ and $\eta$ in different regimes in \cite{el2006rate} and \cite{johnstone2001distribution}, here we argue in a different and more complex direction. As with the right edge, given that $\frac{N}{N+a}\rightarrow\gamma\in(0,1)$, we need to establish a scaling $\tau$ for the left edge of $L_{a,N}$ to make a concrete comparison. Recall that scaling $\tau$ for the left soft edge depends both on the expected center $\mu$ and deviation $\sigma$. We begin our argument with the construction of $\mu$ and $\sigma$, based on the following Liouville--Green's method, introduced by Olver in \cite{olver1997asymptotics} and applied by El Karoui \cite{el2006rate} and Johnstone \cite{johnstone2001distribution} to their LUE right edge estimation.

\subsection{Liouville--Green's Method}\label{sec: L-G method}
We denote $N+a$ as $n$ in the following. Given an $LUE$ $L_{a,N}$, introduce parameters $\lambda_N:=\frac{a}{2}$, $\kappa_N:=N+\frac{a+1}{2}$ and $\omega_N:=\frac{2\lambda_N}{\kappa_N}$, then the function $w_N(y)=y^{\frac{a+1}{2}}e^{-\frac{y}{2}}L^a_N(y)=y^\frac{1}{2}f_{a,N}(y)$, which serves as the primary term in our $\xi$ and $\eta$ in \eqref{LUE int2}, satisfies:
\begin{equation}\label{Lequa}
    \frac{d^2w_N(y)}{dy^2}=\left(\frac{1}{4}-\frac{\kappa_N}{y}+\frac{\lambda_N^2-\frac{1}{4}}{y^2}\right)w_N(y)
\end{equation}
according to the Laguerre recurrences in \cite[Chapter $5.1$]{szego1975orthogonal}. After changing the variable: $z:=\frac{y}{\kappa_N}$, \eqref{Lequa} can be written as:
\begin{equation}\label{W-equa1}
    \frac{d^2w_N(z)}{dz^2}=\left(\kappa_N^2f(z)+g(z)\right)w_N(z)
\end{equation}
with
\begin{equation*}
    f(z)=\frac{(z-z_1)(z-z_2)}{4z^2},~~g(z)=-\frac{1}{4z^2}.
\end{equation*}
Here $z_1=2-\sqrt{4-\omega_N^2}$ and $z_2=2+\sqrt{4-\omega_N^2}$ are called \textbf{turning points} of $f$. As $\omega_N\rightarrow2\left(\frac{1-\gamma}{1+\gamma}\right)$, the parameter $\omega_N$ is in $[\delta,4-\delta]$ for some $\delta>0$ when $N$ is large.

Restrict attention to the equation:
\begin{equation}\label{W-equa1tr}
    \frac{d^2w(z)}{dz^2}=\left(\kappa_N^2f(z)+g(z)\right)w(z)
\end{equation}
with $f$ and $g$ given before. It follows that $w_N(\kappa_Nz)$ is a solution of \eqref{W-equa1tr} from the previous scaling.

The philosophy of the L--G method is to introduce a new variable $\zeta$ in order to further rewrite \eqref{W-equa1tr} as a classical differential equation for $\zeta$. To this end, we aim for $\zeta$ to satisfy $\zeta(\zeta')^2=f$, where the prime denotes differentiation with respect to $z$. This condition ensures that the term involving $f$ in \eqref{W-equa1tr} is killed. The following proposition provides a construction of such a $\zeta$.
\begin{proposition}\label{L-G Trans}
   Given the differential equation \eqref{W-equa1tr}, let $I_1=(a_1,b_1)$, $I_2=(a_2,b_2)$ be two intervals containing exactly one turning point of $f$, say, $z_1\in I_1$ and $z_2\in I_2$. Define $\zeta$ by
\begin{equation}\label{LG1}
\left\{
\begin{array}{ll}
  \frac{2}{3}\zeta^{\frac{3}{2}}(z)= \int_{z}^{z_1}f^{\frac{1}{2}}(t)dt,~z\le z_1\\\\
  \frac{2}{3}(-\zeta)^\frac{3}{2}(z)=\int_{z_1}^z(-f)^\frac{1}{2}(t)dt,~z>z_1
\end{array}
 \right. 
\end{equation}
on $I_1$, and similarly,
\begin{equation}\label{LG2}
\left\{
\begin{array}{ll}
  \frac{2}{3}\zeta^{\frac{3}{2}}(z)= \int_{z_2}^{z}f^{\frac{1}{2}}(t)dt,~z>z_2\\\\
  \frac{2}{3}(-\zeta)^\frac{3}{2}(z)=\int_{z}^{z_2}(-f)^\frac{1}{2}(t)dt,~z\le z_2
\end{array}
 \right. 
\end{equation}
on $I_2$. Then for any $z\in I_j$ where $j\in\{1,2\}$, $\zeta':=\frac{d\zeta}{dz}$ is positive when $z>z_j$ and negative when $z<z_j$. Moreover, for any $z\in I_j$, define $W:=(\zeta')^{\frac{1}{2}}w$ when $z>z_j$ and $W:=(-\zeta')^\frac{1}{2}w$ when $z\le z_j$. Further define $\tilde{f}:=\frac{f}{\zeta}$, which equals $(\zeta')^2$. Then the previous equation \eqref{W-equa1tr} can be expressed as:
\begin{equation}\label{W-equa2}
    \frac{d^2W(\zeta)}{d\zeta^2}=\left(\kappa_N^2\zeta+\Psi(\zeta)\right)W(\zeta)
\end{equation}
where
\begin{equation*}
  \Psi(\zeta)=\tilde{f}^{-\frac{1}{4}}\frac{d^2}{d\zeta^2}\left(\tilde{f}^{\frac{1}{4}}\right)+\frac{g}{\tilde{f}}.
\end{equation*}
This $\zeta$ is called the \textbf{Liouville--Green~transformation} of $f$.
\end{proposition}
\begin{proof}
    The sign of $\zeta'$ and the identity $\zeta(\zeta')^2=f$ are straightforward by the chain rule according to \eqref{LG1} and \eqref{LG2}. With respect to $\zeta$, since $\frac{dw_(z)}{d\zeta}=\frac{dw_(z)}{dz}\frac{dz}{d\zeta}=\frac{w'}{\zeta'}$, it follows that
    \begin{equation}\label{W-equa1'}
            \frac{d^2 w}{d\zeta^2}=\frac{d}{d\zeta}\left(\frac{{w}'}{\zeta'}\right)=\frac{d}{d z}\left(\frac{{w}'}{\zeta'}\right)\frac{dz}{d\zeta}=\frac{w''\zeta'-w'\zeta''}{(\zeta')^3}=\frac{(\kappa_N^2f+g)w}{(\zeta')^2}-\frac{\zeta''}{(\zeta')^2}\frac{dw}{d\zeta}
    \end{equation}
    by \eqref{W-equa1tr}. Now consider the equation for $W$. We only prove the case where $z\le z_1$ as the method applies to other regions by simply substituting $-\zeta'$ by $\zeta'$ if necessary. Since
    \begin{equation*}
        \frac{dW}{d\zeta}=\frac{d\left(\sqrt{-\zeta'}\right)}{dz}\frac{dz}{d\zeta}w+\sqrt{-\zeta'}\frac{dw}{d\zeta}=\frac{\zeta''}{2(\sqrt{-\zeta'})^3}w+\sqrt{-\zeta'}\frac{dw}{d\zeta},
    \end{equation*}
   the equation \eqref{W-equa1'} yields:
    \begin{equation}\label{W-equa2'}
        \begin{split}
            \frac{d^2W}{d\zeta^2}&=\frac{d}{dz}\left(\frac{\zeta''}{2(-\zeta')^\frac{3}{2}}\right)\frac{dz}{d\zeta}w+\frac{\zeta''}{2(-\zeta')^\frac{3}{2}}\frac{dw}{d\zeta}+\frac{d\sqrt{-\zeta'}}{dz}\frac{dz}{d\zeta}\frac{dw}{d\zeta}+\sqrt{-\zeta'}\frac{d^2w}{d\zeta^2}\\
            &=\left(\frac{-\zeta'''}{2(-\zeta')^\frac{5}{2}}-\frac{3(\zeta'')^2}{4(-\zeta')^\frac{7}{2}}\right)w+\frac{\zeta''}{(-\zeta')^\frac{3}{2}}\frac{dw}{d\zeta}+\frac{(\kappa^2_Nf+g)w}{(-\zeta')^\frac{3}{2}}-\frac{\zeta''}{(-\zeta')^\frac{3}{2}}\frac{dw}{d\zeta}\\
            &=\left(\kappa_N^2\zeta+\frac{g}{\tilde{f}}+\frac{\zeta'''}{2(\zeta')^3}-\frac{3(\zeta'')^2}{4(\zeta')^4}\right)W
        \end{split}
    \end{equation}
    where the fact $\frac{f}{(\zeta')^2}=\zeta$ has been used in the third line. Noting that $\tilde{f}^\frac{1}{4}=\sqrt{-\zeta'}$, from the computations in \eqref{W-equa2'}, it is easy to see
    \begin{equation*}
        \begin{split}
            \tilde{f}^{-\frac{1}{4}}\frac{d^2}{d\zeta^2}\left(\tilde{f}^\frac{1}{4}\right)=\frac{1}{\sqrt{-\zeta'}}\frac{d^2(\sqrt{-\zeta'})}{d\zeta^2}=\frac{1}{\sqrt{-\zeta'}}\frac{d}{dz}\left(\frac{\zeta''}{2(-\zeta')^\frac{3}{2}}\right)\frac{1}{\zeta'}=\frac{\zeta'''}{2(\zeta')^3}-\frac{3(\zeta'')^2}{4(\zeta')^4}
        \end{split}
    \end{equation*}
    and hence, \eqref{W-equa2'} is exactly the same as \eqref{W-equa2}.
\end{proof}

Given the fact that our focus is only on the extreme eigenvalues of the LUE, it suffices for our purposes to consider the local existence of $\zeta$ around $z_1$ and $z_2$, which correspond the least and largest eigenvalue, respectively. In other words, Proposition \ref{L-G Trans} directly provides us with the expansion around the extreme eigenvalues via the solution derived from $W$ below. For convention, we simply set $I_1=(0,2)$ and $I_2=(2,\infty)$ in our argument.

Removing $\Psi$ in \eqref{W-equa2} leads to a new equation:
\begin{equation}\label{W-equa3}
    \frac{d^2W(\zeta)}{d\zeta^2}=\kappa^2_N\zeta W(\zeta).
\end{equation}
Define $t:=\kappa_N^\frac{2}{3}\zeta$. Then \eqref{W-equa3} is equivalent to:
\begin{equation*}
    \frac{d^2W(t)}{dt^2}=\frac{d^2W}{d\zeta^2}\kappa_N^{-\frac{4}{3}}=\kappa_N^\frac{2}{3}\zeta W=tW(t),
\end{equation*}
which is exactly the Airy equation with two independent solutions $Ai(t)$ and $Bi(t)$. Since only $Ai$ vanishes at $\infty$, which matches the behavior of our target---the Laguerre function, it is reasonable to directly exclude $Bi(t)$ in our expansion. Consequently, the solution of \eqref{W-equa2} can be written as $W_2(\zeta)=Ai(t)+\tilde{{\epsilon}}_2(t)=Ai\left(\kappa_N^\frac{2}{3}\zeta\right)+\epsilon_2(\zeta)$ for some function $\epsilon_2$ controlled by $\Psi$ (explicit bounds will be given later).
Further notice that $w_2(\kappa_N,z)=\tilde{f}^{-\frac{1}{4}}W_2(\kappa_N,z)$ and $w_N(\kappa_Nz)$ satisfy the same differential equation \eqref{W-equa1tr}. Thus, $w_2$ must be proportional to $w_N(\kappa_Nz)$:
\begin{equation}\label{w2}
   w_N(\kappa_Nz)=c^{\{k\}}_{N}w_2(\kappa_N,z)=c^{\{k\}}_{N}\tilde{f}^{-\frac{1}{4}}(z)\left(Ai\left(\kappa_{N}^{\frac{2}{3}}\zeta^{\{k\}}\right)+\epsilon_2(\kappa_N,z)\right)
\end{equation}
for a constant $c^{\{k\}}_{N}$, which will be computed later, and L--G transform $\zeta^{\{k\}}$ corresponding to different turning points $z_k$. The identity \eqref{w2} illustrates the possibility of large $N$ expansion for $w_N$ towards $Ai$, but only around turning points $z_1$ and $z_2$. Consequently, a natural choice of the $\mu_k$, the expected center of the expansion, is simply $\kappa_N z_k$. Straightforward algebra yields:
\begin{equation*}
    \begin{split}
       \mu_1&=\kappa_Nz_1=(n+N+1)\frac{2-\sqrt{4-\omega_N^2}}{2}=\left(\sqrt{n+\frac{1}{2}}-\sqrt{N+\frac{1}{2}}\right)^2
    \end{split}
\end{equation*}
and similarly, $\mu_2=\left(\sqrt{n+\frac{1}{2}}+\sqrt{N+\frac{1}{2}}\right)^2$.

Now we focus on the deviation $\sigma_k$ around different $z_k$. Notice if we set $\sigma_k=\frac{\kappa_N^{\frac{1}{3}}}{\zeta'(z_k)}$, then
\begin{equation*}
    \kappa_N^{\frac{2}{3}}\zeta(z)=\kappa_N^{\frac{2}{3}}\zeta\left(z_k+\sigma_i\kappa_N^{-1}s\right)\sim\sigma_i\kappa_N^{-\frac{1}{3}}\zeta'(z_k)s=s
\end{equation*}
when $N$ is large since the turning points satisfy $\zeta(z_k)=0$. Consequently, a concrete comparison is possible since the induced Airy function in \eqref{w2} satisfies $Ai\left(\kappa_{N}^{\frac{2}{3}}\zeta\right)\rightarrow Ai(s)$ for suitable $s$ by continuity. This observation motivates our construction of $\sigma_k$. We now proceed to compute $\sigma_1$ and $\sigma_2$. Taking the derivative of \eqref{LG1} directly, we see $\zeta'(z)=\frac{-1}{2z}\left(\frac{(z-z_1)(z-z_2)}{\zeta(z)}\right)^{\frac{1}{2}}$
and hence, $ -\zeta'(z_1)=\frac{1}{2z_1}\left(\frac{z_1-z_2}{\zeta'(z_1)}\right)^{\frac{1}{2}}$ by l'H{\^o}pital's rule. Then it follows that
\begin{equation*}
    \zeta'(z_1)=\left(\frac{z_1-z_2}{4z_1^2}\right)^{\frac{1}{3}}.
\end{equation*}
Plugging $z_k=\frac{\mu_k}{\kappa_N}$ back, we see
\begin{equation*}
    \sigma_1=-\left(\sqrt{n+\frac{1}{2}}-\sqrt{N+\frac{1}{2}}\right)\left(\frac{1}{\sqrt{N+\frac{1}{2}}}-\frac{1}{\sqrt{n+\frac{1}{2}}}\right)^{\frac{1}{3}}
\end{equation*}
and similarly,
\begin{equation*}
    \sigma_2=\left(\sqrt{n+\frac{1}{2}}+\sqrt{N+\frac{1}{2}}\right)\left(\frac{1}{\sqrt{n+\frac{1}{2}}}+\frac{1}{\sqrt{N+\frac{1}{2}}}\right)^{\frac{1}{3}}.
\end{equation*}
The center parameters $\mu_1,~\mu_2$ and deviation parameters $\sigma_1,~\sigma_2$ above match the choice in \cite{johnstone2001distribution}. For convention, we replace $\sigma_1$ by $-\sigma_1$ to maintain the consistency in sign. Therefore, the left-soft-edged scaling is expressed as $\mu_1-\sigma_1s$, whereas the right-soft-edge scaling takes the form $\mu_2+\sigma_2s$. Recalling that the focus of this section is on the left edge, all of the $\mu_{j,k}$'s and $\sigma_{j,k}$'s below are actually $\mu_1$ and $\sigma_1$ with parameters $j,~k\in\mathbb{N}$. Specifically, they are defined as:
\begin{equation*}
    \begin{array}{ll}
        \mu_{j,k}:=\left(\sqrt{j+\frac{1}{2}}-\sqrt{k+\frac{1}{2}}\right)^2,\\
        \\
        \sigma_{j,k}:=\left(\sqrt{j+\frac{1}{2}}-\sqrt{k+\frac{1}{2}}\right)\left(\frac{1}{\sqrt{k+\frac{1}{2}}}-\frac{1}{\sqrt{j+\frac{1}{2}}}\right)^{\frac{1}{3}}.
    \end{array}
\end{equation*}

\subsection{Operator Decomposition and Rescaling}
Now we turn to the decomposition of $\mathbf{K}_{L_{a,N}}^{LS}$ (the scaled operator around the left edge) based on \eqref{LUE int2}. Recall our goal is to establish large $N$ expansions for $\xi$ and $\eta$ in \eqref{explicit xi and zeta}.
Further define
\begin{equation*}
    F_{n,N}(x):=\sigma_{n,N}^{-\frac{1}{2}}\sqrt{\frac{N!}{n!}}x^{\frac{a+1}{2}}e^{-\frac{z}{2}}L_N^{a}(x)
\end{equation*}
to be the normalization of $w_N$ given in \eqref{Lequa},
then the kernels in \eqref{explicit xi and zeta} take the form:
\begin{equation*}
\begin{split}
        &\xi_{a,N}(x)=(-1)^N\sqrt{\frac{Nn}{2}}\sigma_{n,N-1}^{\frac{1}{2}}F_{n-1,N}(x)\frac{1}{x},\\
        &\eta_{a,N}(x)=(-1)^{N-1}\sqrt{\frac{Nn}{2}}\sigma_{n,N-1}^{\frac{1}{2}}F_{n,N-1}(x)\frac{1}{x}.
\end{split}
\end{equation*} 
Consequently, our aim is to perform the large $N$ expansion for:
\begin{equation*}
    F_{n,N}(x)=\sqrt{\frac{N!}{n!\sigma_{n,N}}}w_N(x)=r^{\{k\}}_{N}\left(\frac{\kappa_N}{\sigma_{n,N}^3}\right)^\frac{1}{6}w_2(\kappa_N,z)
\end{equation*}
where $x=\kappa_Nz$ according to \eqref{w2}. In addition, the previous construction by the L--G method indicates that $(\mu_{n,N},\sigma_{n,N})$ is a standard scaling pair with respect to $F_{n,N}$. Similarly, this fact implies the standard pair $(\mu_{n-1,N},\sigma_{n-1,N})$ for $\xi_{a,N}$ and $(\mu_{n,N-1},\sigma_{n,N-1})$ for $\eta_{a,N}$. This inconsistency will prevent us from the promised order as in Theorem \ref{main L left soft} if we directly use either $(\mu_{n-1,N},\sigma_{n-1,N})$ or $(\mu_{n,N-1},\sigma_{n,N-1})$ as our scaling (see similar argument at right edge in \cite{el2006rate}). Thus, for the reasoning of the rescaling parameter $\tau_N$ at the hard edge in \cite{cai2025mesoscopicratesconvergencehermitian}, based on $(\mu_{n-1,N},\sigma_{n-1,N})$ and $(\mu_{n,N-1},\sigma_{n,N-1})$, we will construct a new rescaling $\tau$ to achieve the ideal rate. To this end, we bring in the restriction of $(\tilde\mu_{n,N},\tilde\sigma_{n,N})$ in terms of large $N$ as following:
\begin{equation}\label{restriction of recaling}
    \begin{array}{cc}
    \tilde{\mu}_{n,N}-\mu_{n-1,N}=O(1),~~\tilde\mu_{n,N}-\mu_{n,N-1}=O(1);\\
    \\
    \frac{\tilde{\sigma}_{n,N}}{\sigma_{n,N-1}}=1+O\left(N^{-1}\right),~~\frac{\tilde{\sigma}_{n,N}}{\sigma_{n,N-1}}=1+O\left(N^{-1}\right).
    \end{array}
\end{equation}
\begin{proposition}\label{LUE left soft fac}
    Let $\xi_{a,N}$, $\eta_{a,N}$ be the kernel functions defined in \eqref{explicit xi and zeta} and $(\tilde{\mu}_{n,N},\tilde{\sigma}_{n,N})$ be a pair with respect to the left-soft-edged scaling $\tau$ satisfying the restriction \eqref{restriction of recaling}. Define:
    \begin{equation}\label{left soft kernel}
    \begin{array}{cc}
          K^{LS}_{\tau}(x,y):=\tilde\sigma_{n,N}K_{L_{a,N}}\left(\tilde{\mu}_{n,N}-\tilde\sigma_{n,N}x,\tilde\mu-\tilde\sigma_{n,N}y\right),\\
          \\
        \xi_{\tau}(s):=\tilde\sigma_{n,N}\xi_{{a,N}}\left(\tilde{\mu}_{n,N}-\tilde\sigma_{n,N}s\right),\\
        \\
          \eta_{\tau}(s):=\tilde\sigma_{n,N}\eta_{{a,N}}\left(\tilde{\mu}_{n,N}-\tilde\sigma_{n,N}s\right).
    \end{array}
    \end{equation}
    Then the corresponding operators satisfy: $\mathbf{K}_{\tau}^{LS}=\mathbf{G}_\tau\mathbf{H}_\tau+\mathbf{H}_\tau\mathbf{G}_\tau$ with kernel functions $G_{\tau}(x,y)=\xi_\tau(x+y)$ and $H_\tau(x,y)=-\eta_\tau(x+y)$. Moreover, let $Ai(x+y)$ be the kernel of the integral operator $\mathbf{Ai}$, then the decomposition: $\mathbf{Ai}^2=\mathbf{K}_{Airy}$ holds for the Airy operator.
    
    In particular, for any $s_0\in\mathbb{R}$, $\mathbf{G}_{\tau}$, $\mathbf{H}_\tau$ and $\mathbf{Ai}$ can be regarded as Hilbert--Schmidt integral operators on $L^2(s_0,\infty)$ with kernels $\xi_\tau(x+y-s_0)$, $-\eta_\tau(x+y-s_0)$ and $Ai(x+y-s_0)$ respectively.
\end{proposition}
\begin{proof}
   Under the assumption that $\frac{N}{n}\rightarrow\gamma\in(0,1)$, the parameter satisfies $a>2$ for large $N$ and hence, Proposition \ref{LUE fac} applies. Performing the linear scaling and then shifting by $s_0$ finishes the proof.
\end{proof}

In particular, the kernels $\xi_\tau$ and $\eta_\tau$ take the following explicit forms in terms of $F$:
\begin{equation*}
    \xi_{\tau}(s)=\frac{(-1)^N}{\sqrt2}\left(\frac{\sqrt{Nn}\sigma_{n-1,N}^\frac{1}{2}\tilde{\sigma}_{n,N}}{\mu_{n-1,N}}\right)F_{n-1,N}\left(\tilde{\mu}_{n,N}-\tilde{\sigma}_{n,N}s\right)\left(\frac{\mu_{n-1,N}}{\tilde{\mu}_{n,N}-\tilde{\sigma}_{n,N}s}\right),
\end{equation*}
\begin{equation*}
    \eta_{\tau}(s)=\frac{(-1)^{N-1}}{\sqrt2}\left(\frac{\sqrt{Nn}\sigma_{n,N-1}^\frac{1}{2}\tilde{\sigma}_{n,N}}{\mu_{n,N-1}}\right)F_{n,N-1}\left(\tilde{\mu}_{n,N}-\tilde{\sigma}_{n,N}s\right)\left(\frac{\mu_{n,N-1}}{\tilde{\mu}_{n,N}-\tilde{\sigma}_{n,N}s}\right).
\end{equation*}

The following lemma regarding the constant $c_N^{\{1\}}$ in \eqref{w2} will allow us to derive the closed form of $F$.
\begin{lemma}\label{c and r}
    Suppose $c^{\{1\}}_{N}$ is the constant defined in \eqref{w2} with respect to $z_1$. Then
 \begin{equation}\label{explicit4c}
     c^{\{1\}}_{N}=\frac{\sqrt{2\pi a}(N+a)!\kappa_N^\frac{1}{6}}{a!N!}a^ae^{-\frac{a}{2}}\left(N+\frac{1}{2}\right)^{N+\frac{1}{2}}\left(n+\frac{1}{2}\right)^{-\left(n+\frac{1}{2}\right)}.
 \end{equation}
 Moreover, define $r^{\{1\}}_{N}:=\sqrt{\frac{N!}{(N+a)!}}c^{\{k\}}_{N}\kappa_N^{-\frac{1}{6}}$. Then for large $N$,
 \begin{equation}\label{B4rN}
     r^{\{1\}}_{N}=1+O(N^{-1}).
 \end{equation}
\end{lemma}
\begin{proof}
    See Appendix \ref{P4CandR}.
\end{proof}
\bigskip\noindent\textbf{Remark.} Explicit expression of $c^{\{2\}}_{N}$ and $r^{\{2\}}_{N}$ are given in \cite{johnstone2001distribution}. Since $c^{\{2\}}_{N}$ carries a sign of $(-1)^N$, it directly offsets the signs of $\xi_\tau$ and $\eta_\tau$, thereby eliminating the need of pairing as in Theorem \ref{main left soft edge2} in the largest eigenvalue expansion. In contrast, because $c^{\{1\}}_{N}>0$ in the previous lemma and offers no such sign cancellation, classification according the the parities of $N$ becomes necessary in our least eigenvalue expansion.

Given our focus is solely on $z_1$, we will drop the notation $\{1\}$ in the following argument.

For any $j,~k\in\mathbb{N}$, define the deviation parameter with respect to the scaling $\tau$ as:
 \begin{equation*}
     \theta_{j,k}:=(nN)^\frac{1}{4}\sqrt{\sigma_{j,k}}\frac{\tilde{\sigma}_{n,N}}{\tilde{\mu}_{n,N}}.
 \end{equation*}
 Then both $\xi_\tau$ and $-\eta_\tau$ can be written as:
\begin{equation}\label{F2h}
   \frac{(-1)^N}{\sqrt2}\theta_{j,k}F_{j,k}(\tilde{x}_N(s))\frac{\mu_{j,k}}{\tilde{x}_N(s)}:=\frac{(-1)^N\theta_{j,k}}{\sqrt2}h_{j,k}(\tilde{x}_N(s))
\end{equation}
 for some $j,~k$ and $\tilde{x}_N(s):=\tilde\mu_{n,N}-\tilde\sigma_{n,N}s$ since the sign of $\eta_\tau$ is $(-1)^{N-1}$.
 
 The following lemma indicates a simple fact that $\theta$ is small given the scaling $\tau$ is closed to both $(\mu_{n-1,N},\sigma_{n-1,N})$ and $(\mu_{n,N-1},\sigma_{n,N-1})$.
 \begin{lemma}\label{order}
    If $\tau:(\tilde{\mu}_{n,N},\tilde{\sigma}_{n,N})$ is a left-soft-edged scaling satisfying \eqref{restriction of recaling}, then for large $N$,
    \begin{equation*}
        \theta_{n-1,N},~~\theta_{N,n-1}=1+O(N^{-1}).
    \end{equation*}
\end{lemma}
\begin{proof}
We only prove $\theta_{n-1,N}=1+O\left(N^{-1}\right)$ since the method simply applies to $\theta_{n,N-1}$. Recalling that $\tilde{\mu}_{n,N}=\mu_{n-1,N}+O\left(N^{-1}\right)$ and $\frac{\tilde{\sigma}_{n,N}}{\sigma_{n-1,N}}=1+O\left(N^{-1}\right)$ according to \eqref{restriction of recaling}, we see
\begin{equation*}
    \frac{\mu_{n-1,N}}{\tilde{\mu}_{n,N}},~~\frac{\tilde{\sigma}_{n,N}}{\sigma_{n-1,N}}=1+O\left(N^{-1}\right).
\end{equation*}
Consequently,
\begin{equation}\label{Deviation1}
    \begin{split}
        \theta_{n-1,N}&=(nN)^\frac{1}{4}\sigma_{n-1,N}^\frac{1}{2}\frac{\tilde{\sigma}_{n,N}}{\tilde{\mu}_{n,N}}=\frac{\tilde{\sigma}_{n,N}}{\sigma_{n-1,N}}\frac{\mu_{n-1,N}}{\tilde{\mu}_{n,N}}\frac{(nN)^\frac{1}{4}\sigma_{n-1,N}^\frac{3}{2}}{\mu_{n-1,N}}\\
        &=\frac{\tilde{\sigma}_{n,N}}{\sigma_{n-1,N}}\frac{\mu_{n-1,N}}{\tilde{\mu}_{n,N}}\left(\frac{nN}{(n-\frac{1}{2})(N+\frac{1}{2})}\right)^\frac{1}{4}\\
        &=\frac{\tilde{\sigma}_{n,N}}{\sigma_{n-1,N}}\frac{\mu_{n-1,N}}{\tilde{\mu}_{n,N}}\left(1-\frac{1}{2n}\right)^{-\frac{1}{4}}\left(1+\frac{1}{2N}\right)^{-\frac{1}{4}}\\
        &=1+O\left(N^{-1}\right)
    \end{split}
\end{equation}
since all terms of the product in the third line are $1+O\left(N^{-1}\right)$.
\end{proof}
The concrete construction of such $\tau:(\tilde{\mu}_{n,N},\tilde{\sigma}_{n,N})$ will be given in Lemma \ref{re4D} below.

\subsection{Proof of the Left Soft Edge for LUE}
With necessary parameters and functions established in our previous argument, this section proves the following important theorem regarding $\xi_\tau$, $\eta_\tau$ and $Ai$ on $L^2(s_0,\infty)$.

\begin{theorem}\label{main left soft edge2}
 Consider an LUE $L_{a,N}$ such that $\lim\limits_{N\rightarrow\infty}\frac{N}{n}=\gamma\in(0,1)$ where $n=N+a$. On any $A\subset[s_0,+\infty)$ for $s_0>-\infty$. Define parameters:
    \begin{equation*}
      \begin{array}{ll}
      \mu_{k,m}:=\left(\sqrt{k+\frac{1}{2}}-\sqrt{m+\frac{1}{2}}\right)^2,\\
      \sigma_{k,m}:=\left(\sqrt{k+\frac{1}{2}}-\sqrt{m+\frac{1}{2}}\right)\left(\left(m+\frac{1}{2}\right)^{-\frac{1}{2}}-\left(k+\frac{1}{2}\right)^{-\frac{1}{2}}\right)^{\frac{1}{3}}.
      \end{array}
    \end{equation*}
    Let $G_\tau=\xi_\tau$, $H_\tau=-\eta_\tau$ be kernels defined in \eqref{left soft kernel} and $\tau$ be the scaling with respect to
\begin{equation*}
    \tilde{\mu}_{n,N}=\left(\frac{1}{\sqrt{\sigma_{n-1,N}}}+\frac{1}{\sqrt{\sigma_{n,N-1}}}\right)\left(\frac{1}{\mu_{n-1,N}\sqrt{\sigma_{n-1,N}}}+\frac{1}{\mu_{n,N-1}\sqrt{\sigma_{n,N-1}}}\right)^{-1},
\end{equation*}
\begin{equation*}
    \tilde{\sigma}_{n,N}=\left(\frac{\sqrt{\sigma_{n-1,N}}}{\mu_{n-1,N}}+\frac{\sqrt{\sigma_{n,N-1}}}{\mu_{n,N-1}}\right)\left(\frac{1}{\mu_{n-1,N}\sqrt{\sigma_{n-1,N}}}+\frac{1}{\mu_{n,N-1}\sqrt{\sigma_{n,N-1}}}\right)^{-1}.
\end{equation*}
Then, for large $N$ and $t\in A$,
   \begin{equation}\label{main2'}
\begin{array}{ll}
  |G_{\tau}(t)+H_{\tau}(t)+(-1)^{N+1}\sqrt2 Ai(t)|\le N^{-\frac{2}{3}}C(s_0,\gamma)e^{-\frac{t}{2}},\\
  \\
  \left|G_{\tau}(t)+(-1)^{N+1}\frac{\sqrt2}{2} Ai(t)\right|\le N^{-\frac{1}{3}}C(s_0,\gamma)e^{-\frac{t}{2}},\\
  \\
  \left|H_{\tau}(t)+(-1)^{N+1}\frac{\sqrt2}{2} Ai(t)\right|\le N^{-\frac{1}{3}}C(s_0,\gamma)e^{-\frac{t}{2}}.
\end{array}
\end{equation}
Furthermore, for any fixed $\gamma$, $C(s_0,\gamma)$ is non-increasing with respect to $s_0$.
\end{theorem}

\bigskip\noindent\textbf{Remark.} In \cite{el2006rate}, El Karoui demonstrated that for the largest eigenvalue, the dependence on $\gamma$ in $C(s_0,\gamma)$ can be eliminated by picking a uniform bound with respect to $\gamma$. However, this approach does not extend to the least one, as both the lower bound $0$ and upper bound $1$ play a role in our setting. Thus, a uniform bound in $\gamma$ can not be established for each estimate, and all bounds necessarily depend on $\gamma$. However, in a slightly different yet natural setting, we can weaken the dependence on $\gamma$ as following.
\begin{theorem}\label{LUE soft left2}
    Consider an LUE $L_{a,N}$. Suppose there exist $\Theta_1,~\Theta_2\in \mathbb{R}$ such that $0<\Theta_1\le\frac{N}{n}\le \Theta_2<1$ and $\lim\limits_{N\rightarrow\infty}\frac{N}{n}=\gamma\in(0,1)$ where $n=N+a$. Then on any $A\subset[s,+\infty)$ for $s>-\infty$, with the same parameters as in Theorem \ref{LUE left soft}, for $N\in\mathbb{N}$ large enough,
    \begin{equation}\label{soft2}
W_1\left(\mathcal{N}_{L_{a,N}}^{LS},\mathcal{N}_{Ai}\right)\le\frac{g(s,\Theta_1,\Theta_2)e^{-s}}{N^{\frac{2}{3}}}
    \end{equation}
 where $g$ is nonincreasing with respect to $s$ for any fixed $\Theta_1$ and $\Theta_2$. 
\end{theorem}

Before we proceed to the proof of Theorem \ref{LUE left soft}, we introduce some auxiliary functions and bounds given in \cite[Chapter $11$]{olver1997asymptotics} involving the Airy function.
\begin{proposition}\label{bound for Airy function}\cite[Section $11.1\&11.2$]{olver1997asymptotics} The Airy function $Ai(x)$ given in Definition \ref{def4 classic DPPs} is the unique solution of the {Airy equation}: $y''(x)=xy(x)$ vanishing at $+\infty$. On $\mathbb{R}^+$, $Ai(x)$ is a decreasing function such that
\begin{equation}\label{ExB4Airy}
\begin{split}
        &0<Ai(x)\le\frac{1}{2\sqrt{\pi}x^{\frac{1}{4}}}\exp\left(-\frac{2}{3}x^\frac{3}{2}\right),\\&Ai(x)\sim\frac{1}{2\sqrt{\pi}x^{\frac{1}{4}}}\exp\left(-\frac{2}{3}x^\frac{3}{2}\right)~for~x~large.
\end{split}
\end{equation}
  Moreover, its derivative $Ai'(x)$, is an increasing negative function on $\mathbb{R}^+$ according to the Airy equation: $Ai''(x)=xAi(x)\ge0$ and hence, $|Ai'(x)|$ is a decreasing function such that
  \begin{equation}\label{ExB4Airy'}
      |Ai'(x)|\le \frac{x^\frac{1}{4}}{2\sqrt\pi}\exp\left(-\frac{2}{3}x^\frac{3}{2}\right).
  \end{equation}
\end{proposition}

\begin{proposition}\label{aux4AiandBi}
    (\cite[Section $11.2$]{olver1997asymptotics}) Let $Bi(x)$ be the other solution of the Airy equation and $c$ be the largest negative root of $Ai(x)=Bi(x)$ (see \cite[Section $11.2.2$]{olver1997asymptotics} for the explicit expression for both $Bi$ and $c$). Define the modulus function $\mathbf{M}$ and weight function $\mathbf{E}$ as:
    \begin{equation*}
        \mathbf{E}(x):=\left\{\begin{array}{cc}
\sqrt{\frac{Bi(x)}{Ai(x)}},~~x>c\\
1,~~~x\le c
        \end{array}\right.~~and~~ \mathbf{M}(x):=\sqrt{\mathbf{E}^2(x)Ai^2(x)+\frac{Bi^2(x)}{\mathbf{E}^2(x)}}.
    \end{equation*}
    Then $|Ai(x)|\le\left|\frac{\mathbf{M}(x)}{\mathbf{E}(x)}\right|$, and $\mathbf{E}$ increasing on $(0,\infty)$. Moreover, as $x\rightarrow\infty$,  $\mathbf{E}(x)\sim\frac{1}{\sqrt2}\exp\left(\frac{2}{3}x^\frac{3}{2}\right)$ and $\mathbf{M}(x)\sim\frac{1}{\pi^\frac{1}{2}x^\frac{1}{4}}$. Introduce the phase parameter $\theta(x):=\arctan\left(\frac{\mathbf{E}^2(x)Ai(x)}{Bi(x)}\right)$. Then $Ai$ and $Bi$ are related by:
    \begin{equation}\label{MandEzero}
        Ai(x)=\frac{\mathbf{M}(x)}{\mathbf{E}(x)}\sin\theta(x),~~Bi(x)=\mathbf{M}(x)\mathbf{E}(x)\cos\theta(x).
    \end{equation}
    In terms of $Ai'$ and $Bi'$, further define a new modulus function $\mathbf{N}$ and phase function $\omega(x)$ by:
        \begin{equation*}
        \begin{split}
                   & \mathbf{N}(x):=\left\{\begin{array}{cc}
\sqrt{\frac{Ai'^2(x)Bi^2(x)+Ai^2(x)Bi'^2(x)}{Ai(x)Bi(x)}},~~x>c\\
\sqrt{Ai'^2(x)+Bi'^2(x)},~~x\le c
        \end{array}\right.\\&   \omega(x):=\left\{\begin{array}{cc}
           \arctan\left(\frac{Ai'(x)Bi(x)}{Ai(x)Bi'(x)}\right),~~x>c\\
           \arctan\left(\frac{Ai'(x)}{Bi'(x)}\right),~~x\le c
        \end{array}\right..
        \end{split}
    \end{equation*}
    Then $Ai'$ and $Bi'$ are related by:
    \begin{equation}\label{NandEfirst}
         Ai'(x)=\frac{\mathbf{N}(x)}{\mathbf{E}(x)}\sin\omega(x),~~Bi'(x)=\mathbf{N}(x)\mathbf{E}(x)\cos\omega(x).
    \end{equation}
\end{proposition}

The following simple bound derived from the Laguerre recurrence in \cite{szego1975orthogonal} will be used when dealing with the large $x$ region for $F_{n,N}(x)$ later.
\begin{lemma}\label{B4LagPoly}
For any $n,~a\in\mathbb{N}$ and $x\ge0$, the Laguerre polynomials are bounded as
\begin{equation}\label{B4Lag'}
        \left|L_n^a(x)\right|\le\left(\begin{matrix}
            n+a\\n
        \end{matrix}\right)e^{\frac{x}{2}}.
\end{equation}
\end{lemma}

According to \eqref{F2h}, the large $N$ expansion for $\xi_\tau$ and $\eta_\tau$ can be established by estimating $h_{j,k}$. Since
\begin{equation*}
     h_{n,N}(x)=F_{n,N}(x)\frac{\mu_{n,N}}{x}
     =r_N\left(\frac{\kappa_N}{\sigma_{n,N}^3}\right)^{\frac{1}{6}}\frac{\mu_{n,N}}{x}\tilde{f}^{-\frac{1}{4}}(z)\left(Ai\left(\kappa_N^{\frac{2}{3}}z\right)+\epsilon_2(\kappa_N,z)\right)
\end{equation*}
according to the L$-$G method, handling each term on the far right regarding ideal order yields the needed expansion. The term $r_N$ can be dropped without any order problem since $r_N=1+O\left(N^{-1}\right)$ by Lemma \ref{c and r}. Now consider the difference between $(-1)^Nh_{n,N}$ and $Ai(s)$. We assume $N$ is odd because the even case is similar. Let $x_N(s):=\mu_{n,N}-\sigma_{n,N}s$ be the ``standard" scaling with respect to $F_{n,N}$. Then
\begin{equation*}
\begin{split}
       \Delta_{n,N}(x_N(s))&:=|-\theta_{n,N}h_{n,N}(x_N(s))+Ai(s)|\\&\le \theta_{n,N}|r_N Ai(s)-h_{n,N}(x_N(s))|+|1-r_N \theta_{n,N}||A(s)|\\
       &:=\theta_{n,N}\Delta'_{n,N}(s)+\Delta''_{n,N}(s).
\end{split}
\end{equation*}
Since $\theta_{n-1,N}=1+O\left(N^{-1}\right)$ by Lemma \ref{order}, we can replace $\theta_{n,N}$ by $1$ without causing any order problem. Moreover, the boundedness of $Ai(s)$ itself (see \eqref{bound for Airy function}) indicates $N\Delta''_{n,N}(s)\le C(s_0)e^{-\frac{s}{2}}$ for the second term. Thus, only the term $\Delta'_{n,N}$ matters in our argument. Further define:
\begin{equation*}\begin{split}
     \Delta'_{n,N}(s)&\le \left|h_{n,N}(x_N(s))-r_NAi\left(\kappa_N^\frac{2}{3}\zeta(x_N(s))\right)\right|+r_N\left|Ai\left(\kappa_N^\frac{2}{3}\zeta(x_N(s))\right)-Ai(s)\right|\\
      &:= B_{n,N}(x_N(s))+r_N|D_{n,N}(x_N(s))|
\end{split}
\end{equation*} and $e_0:=1-\frac{1}{e}\in\left(0,\frac{2}{3}\right)$.
We have the following results involving $B_{n,N}$ and $|D_{n,N}|$.
\begin{lemma}\label{keyI}
    Given a fixed $s_0\in\mathbb{R}$, assume the scaling $(\tilde{\mu}_{n,N},\tilde\sigma_{n,N})$ satisfies the order restriction in \eqref{restriction of recaling}. Then under the assumption that $\frac{N}{n}\rightarrow\gamma\in(0,1)$, we have:
    \begin{equation*}
        N^{\frac{2}{3}}B_{n,N}(x_N(s))\le C(s_0,\gamma)e^{-\frac{s}{2}}
    \end{equation*}
    for any $s\in\left[s_0,\frac{e_0\tilde{\mu}_{n,N}}{\tilde{\sigma}_{n,N}}\right]$ when $N$ is large enough. Furthermore, with  $\tilde{x}_N(s):=\tilde{\mu}_{n,N}-\tilde{\sigma}_{n,N}s$,
    \begin{equation*}
        N^{\frac{2}{3}}B_{n-1,N}(\tilde{x}_N(s)),~N^{\frac{2}{3}}B_{n,N-1}(\tilde{x}_N(s))\le C(s_0,\gamma)e^{-\frac{s}{2}}
    \end{equation*}
    for $N$ large.
\end{lemma}
\begin{proof}
   See Appendix \ref{P4I1}. 
\end{proof}

\begin{lemma}\label{keyI2}
    Under the same assumption as in the previous lemma, for any $s\in\left[s_0,\frac{e_0\tilde{\mu}_{n,N}}{\tilde{\sigma}_{n,N}}\right]$, when $N$ is large enough,
    \begin{equation*}
        N^{\frac{2}{3}}|D_{n,N}(x_N(s))|\le C(s_0,\gamma)e^{-\frac{s}{2}}
    \end{equation*}
    where $C$ is a constant determined by $s_0$ and $\gamma$
    \end{lemma}
 \begin{proof}
   See Appendix \ref{P4I2}.  
 \end{proof}
Since the main term of $\xi_\tau$ is $h_{n-1,N}$ and $\eta_\tau$ is $h_{n,N-1}$, Lemma \ref{keyI} implies that the scaling $\tau$ with the conditions as in Lemma \ref{order} does not influence our rate $N^\frac{2}{3}$ for the $B$ part in neither $\xi_{\tau}$ nor $\eta_{\tau}$. However, similar reasoning fails for the $D$ part as in \cite[Section $3.4.2$ ]{el2006rate}. To be more specific, directly replacing $x_{n-1,N}(s)$ or $x_{n,N-1}(s)$ by $\tilde{x}_N(s)$ in Lemma \ref{keyI2} will prevent us from an order higher than $N^\frac{1}{3}$. Thus, a more concrete construction of $\tilde{\mu}_{n,N}$ and $\tilde{\sigma}_{n,N}$ is essential.

Noticing that
\begin{equation*}
\begin{split}
&|2Ai(s)-\theta_{n-1,N}h_{n-1,N}(\tilde{x}_N(s))-\theta_{n,N-1}h_{n,N-1}(\tilde{x}_N(s))|\\&
\le\Delta''_{n-1,N}+\Delta''_{n,N-1}+B_{n-1,N}(\tilde{x}_N(s))+B_{n,N-1}(\tilde{x}_N(s))\\
&~~+r_N|\theta_{N-1,n}D_{N-1,n}(\tilde{x}_N(x))+\theta_{N,n-1}D_{N,n-1}(\tilde{x}_N(s))|
\end{split}
\end{equation*}
and $r_N=1+O(N^{-1})$, to obtain the required order, we aim for $\tilde{\mu}_{n,N}$ and $\tilde{\sigma}_{n,N}$ to satisfy:
\begin{equation*}
\begin{split}
       N^{\frac{2}{3}}D(\tilde{x}_N(s))&:=N^{\frac{2}{3}}|\theta_{N-1,n}D_{N-1,n}(\tilde{x}_N(x))+\theta_{N,n-1}D_{N,n-1}(\tilde{x}_N(s))|\\&\le C(s_0,\gamma)e^{-\frac{s}{2}}
\end{split}
\end{equation*}
when $N$ is large.
The following lemma provides us with such an explicit pair.
\begin{lemma}\label{re4D}
Under the same assumption as in Lemma \ref{keyI}, let $\tilde{x}_N(s)=\tilde{\mu}_{n,N}-\tilde{\sigma}_{n,N}s$ with
\begin{equation}\label{recenter}
    \tilde{\mu}_{n,N}=\left(\frac{1}{\sigma_{n-1,N}^{1/2}}+\frac{1}{\sigma_{n,N-1}^{1/2}}\right)\left(\frac{1}{\mu_{n-1,N}\sigma_{n-1,N}^{1/2}}+\frac{1}{\mu_{n,N-1}\sigma_{n,N-1}^{1/2}}\right)^{-1}
\end{equation}
and
\begin{equation}\label{rescal}
    \tilde{\sigma}_{n,N}=\left(\frac{\sigma_{n-1,N}^{1/2}}{\mu_{n-1,N}}+\frac{\sigma_{n,N-1}^{1/2}}{\mu_{n,N-1}}\right)\left(\frac{1}{\mu_{n-1,N}\sigma_{n-1,N}^{1/2}}+\frac{1}{\mu_{n,N-1}\sigma_{n,N-1}^{1/2}}\right)^{-1}.
\end{equation} 
Then
\begin{equation*}
    N^{\frac{2}{3}}D(\tilde{x}_N(s))\le C(s_0,\gamma)e^{-\frac{s}{2}}
\end{equation*}
holds for any $s\in\left[s_0,\frac{e_0\tilde{\mu}_{n,N}}{\tilde{\sigma}_{n,N}}\right]$ when $N$ large enough.
\end{lemma}
\begin{proof}
    See Appendix \ref{P4re4D}.
\end{proof}

Lemma \ref{keyI}, \ref{keyI2} and \ref{re4D} altogether lead to the promised order and exponential bound on $\left[s_0,\frac{e_0\tilde{\mu}_{n,N}}{\tilde{\sigma}_{n,N}}\right]$, which is the most significant part in our proof. The remaining region, $\left[\frac{e_0\tilde{\mu}_{n,N}}{\tilde{\sigma}_{n,N}},\frac{\tilde{\mu}_{n,N}}{\tilde{\sigma}_{n,N}}\right]$, can be easily handled with the help of the following lemma.

\begin{lemma}\label{inlarge}
    For any $r\ge0$ and $t\in\left[\frac{e_0\tilde{\mu}_{n,N}}{\tilde{\sigma}_{n,N}},\infty\right)$, when $N$ is sufficiently large, we have
    \begin{equation*}
        |G_{\tau}(t)|,~|H_{\tau}(t)|\le N^{-r}C(\gamma,r)e^{-\frac{t}{2}}
    \end{equation*}
    where $C(\gamma,r)$ is a constant determined by $\frac{N}{N+a}\rightarrow\gamma\in(0,1)$ and $r$.
\end{lemma}
\begin{proof}
    See Appendix \ref{P4inlarge}.
\end{proof}

Now the proof of Theorem \ref{main left soft edge2} is clear.
\medbreak\noindent {\itshape Proof of Theorem \ref{main left soft edge2}.}\enspace
    Split the interval into $\left[s_0,\frac{e_0\tilde{\mu}_{n,N}}{\tilde{\sigma}_{n,N}}\right)\cup\left[\frac{e_0\tilde{\mu}_{n,N}}{\tilde{\sigma}_{n,N}},+\infty\right):=J_{1,N}\cup J_{2,N}$ for large $N$. As indicated above, we only consider odd $N$.
As $N^\frac{1}{6}\rightarrow\infty$, $t\ge N^\frac{1}{6}\ge 1$ and $\frac{t^\frac{3}{2}}{12}\ge\frac{2}{3}\log N$ on $J_{2,N}$ for large $N$. Hence, when $N$ is large enough, on $J_{2,N}$, the bound \eqref{bound for Airy function} gives us
      \begin{equation}\label{B4lAi}
       N^\frac{2}{3} Ai(t)\le\exp\left\{\frac{2}{3}\log N-\frac{7}{12}t^\frac{3}{2}\right\}\le \exp\left\{-\frac{1}{2}t^\frac{3}{2}\right\}\le e^{-\frac{t}{2}}.
    \end{equation}
    In addition, on $J_{2,N}$, we already have the bound $N^{-\frac{2}{3}}C\left(\gamma,\frac{2}{3}\right)e^{-\frac{t}{2}}$ for both $|G_{\tau'}(t)|$ and $|H_{\tau'}(t)|$ by Lemma \ref{inlarge}. Combining \eqref{B4lAi}, we see \eqref{main2'} holds on $J_{2,N}$. On $J_{1,N}$, \eqref{main2'} follows from the bounds of $B$ and $D$ in Lemma \ref{keyI} and \ref{re4D}. The monotonicity of $C(s_0,\gamma)$ is easy to see in the proof of those lemmas.\qed

Theorem \ref{main left soft edge2} leads to a key inequality for the corresponding norm.
\begin{corollary}\label{normAi}
Under the same assumption as in Theorem \ref{main left soft edge2}, then on $[s_0,\infty)$, for $N$ large enough, we have
   \begin{equation}\label{main3}
\begin{array}{ll}
  \|G_{\tau}+H_{\tau}+(-1)^{N+1}\sqrt2 Ai\|_{L^2(s_0,\infty)}\le N^{-\frac{2}{3}}C(s_0,\gamma)e^{-\frac{s_0}{2}}\\
  \\
  \left\|G_{\tau}+(-1)^{N+1}\frac{\sqrt2}{2} Ai\right\|_{L^2(s_0,\infty)}\le N^{-\frac{1}{3}}C(s_0,\gamma)e^{-\frac{s_0}{2}}\\
  \\
  \left\|H_{\tau}+(-1)^{N+1}\frac{\sqrt2}{2} Ai\right\|_{L^2(s_0,\infty)}\le N^{-\frac{1}{3}}C(s_0,\gamma)e^{-\frac{s_0}{2}}
\end{array}
\end{equation}
where $C$ is continuous and nonincreasing with respect to $s_0$ for any fixed $\gamma$.
\end{corollary}
\begin{proof}
    We only prove the first inequality for odd $N$ as the other ones follow by the remaining inequalities in Theorem \ref{main left soft edge2}. In fact, denote $t:=x+y-s_0$, when $N$ is large enough, the first inequality in \eqref{main2'} yields:
    \begin{equation*}
        \begin{split}
            \|G_{\tau}+H_{\tau}+\sqrt2 Ai\|^2_{L^2(s_0,\infty)}&=\int_{s_0}^\infty\int_{s_0}^\infty|G_{\tau'}(t)+H_{\tau'}(t)+\sqrt2Ai(t)|^2dxdy\\&\le N^{-\frac{4}{3}}C^2(s_0,\gamma)\int_{s_0}^\infty\int_{s_0}^\infty e^{-x-y+s_0}dxdy\\&=N^{-\frac{4}{3}}C^2(s_0,\gamma)e^{-s_0}.
        \end{split}
    \end{equation*}
    Taking the square root gives us the required result.
\end{proof}

Now we are ready to establish the proof of the key theorems in this paper.
\medbreak\noindent {\itshape Proof of Theorem \ref{LUE left soft}.}\enspace
For any $A\subset I=[s,\infty)$, recall
\begin{equation*}
    \begin{split}
W_1\left(\mathcal{N}_{L_{a,N}}^{LS},\mathcal{N}_{Ai}\right)&\le\|\mathbf{K}_{L_{a,N}}-\mathbf{K}_{Ai}\|_1=\|\mathbf{G}_{\tau}\mathbf{H}_{\tau}+\mathbf{H}_{\tau'}\mathbf{G}_{\tau}-\mathbf{Ai}^2\|_1\\
        &\le\frac{1}{2}\|G_{\tau}+H_{\tau}-\sqrt2 Ai\|_{L^2(s,\infty)}\|G_{\tau}+H_{\tau}-\sqrt2 Ai\|_{L^2(s,\infty)}+\frac{1}{2}\|G_{\tau'}-H_{\tau'}\|^2_{L^2(s,\infty)}.
    \end{split}
\end{equation*}
Combining this with \eqref{main3}, no matter what parity $N$ was, $\|G_{\tau}-H_{\tau}\|^2_{L^2(s,\infty)}$ is always of order $N^{-\frac{2}{3}}$ for large $N$. For the Airy function, continuity of $Ai$ on $\mathbb{R}^+$ and the bound in \eqref{ExB4Airy} indicate the existence of $C>0$ such that
\begin{equation*}
    |Ai(x)|\le\left\{\begin{array}{cc}
      C,~~0\le x\le1\\
      \\
      \exp\left(-\frac{2}{3}x^\frac{3}{2}\right)\le\exp\left(-\frac{x}{2}\right),~~x>1.
    \end{array}\right.
\end{equation*}
Hence, there exists an absolute constant $M>0$ such that $|Ai(x)|\le Me^{-\frac{x}{2}}$ for any $x\in\mathbb{R}^+$. Moreover, together with the bound of $|Ai|$ above, Theorem \ref{main left soft edge2} also illustrates that regardless the parities of $N$, for any $t\in A\subset I$,
\begin{equation*}
\begin{split}
      \left|G_{\tau}(t)+H_{\tau}(t)\right|^2&\le\left(N^{-\frac{2}{3}}C(s,\gamma)e^{-\frac{t}{2}}+\sqrt2|Ai(t)|\right)^2\\&\le C^2(s,\gamma)e^{-t}+M_1 C(s,\gamma)e^{-t}+M_2e^{-t}\\&=C(s,\gamma)e^{-t}
\end{split}
\end{equation*}
holds for some constants $M_1$ and $M_2$. Thus, for $N$ large enough, regardless of the parities,
\begin{equation}\label{boundpartofdif}
    \begin{split}
        &\|G_{\tau}+H_{\tau}-\sqrt2Ai\|_{L^2(s,\infty)},~~\|G_{\tau}+H_{\tau}+\sqrt2Ai\|_{L^2(s,\infty)}\\&\le\|G_{\tau}+H_{\tau}\|_{L^2(s,\infty)}+\sqrt2\|Ai\|_{L^2(s,\infty)}\\
        &\le\left(\int_{s}^\infty\int_{s}^\infty C(s,\gamma)e^{-x-y+s}dxdy\right)^\frac{1}{2}+\left(\int_{s}^\infty\int_{s}^\infty M^2e^{-x-y+s}dxdy\right)^\frac{1}{2}\\&\le C(s,\gamma)
    \end{split}
\end{equation}
where $C(s,\gamma)$ denotes different bounded and non-increasing functions with respect to $s$ for any fixed $\gamma$.

Now, assuming $N$ is large and odd, \eqref{main3} and \eqref{boundpartofdif} yield: $\|G_{\tau}+H_{\tau}+\sqrt2 Ai\|_{L^2(s,\infty)}$ is of order $N^{-\frac{2}{3}}$ with an exponential decay with respect to $s$ and $\|G_{\tau}+H_{\tau}-\sqrt2 Ai\|_{L^2(s,\infty)}$ is bounded. Similarly, if $N$ is large and even, then \eqref{boundpartofdif} leads to: $\|G_{\tau}+H_{\tau}-\sqrt2 Ai\|_{L^2(s,\infty)}$ is of order $N^{-\frac{2}{3}}$ with an exponential decay while $\|G_{\tau}+H_{\tau}+\sqrt2 Ai\|_{L^2(s,\infty)}$ is bounded. As a result, regardless of the parities of $N$, \eqref{main L left soft} holds if $N$ is large.\qed
\bigskip
\medbreak\noindent {\itshape Proof of Theorem \ref{LUE soft left2}.}\enspace
Recalling the fact that $0<\Theta_1\le\frac{N}{n}\le\Theta_2<1$, repeating the previous argument leads to
\begin{equation*}
W_1\left(\mathcal{N}_{L_{a,N}}^{LS},\mathcal{N}_{Ai}\right)\le\frac{g(s,\gamma,\Theta_1,\Theta_2)e^{-s}}{N^{\frac{2}{3}}}
\end{equation*}
for $N\in\mathbb{N}$ large. Since $\gamma\in[\Theta_1,\Theta_2]$, the appraoch from \cite{el2006rate}, combining with our estimates, enables us to eliminate the dependence on $\gamma$ by establishing a uniform bound on $[\Theta_1,\Theta_2]$.\qed

%% file: Appendix.tex
\section{Appendix}
\subsection{Properties of the Function in L--G Transformation}
Recall the function $f$ in the L--G transformation is defined as:
\begin{equation}\label{LG function}
    f(x)=\frac{(x-z_1)(x-z_2)}{4x^2}=\frac{1}{4x^2}\left(x^2-(z_1+z_2)x+z_1z_2\right).
\end{equation}
 We see $f$ is positive on $(-\infty,0)\cup(0,z_1)\cup(z_2,+\infty)$ and negative on $(z_1,z_2)$. Given the fact that our focus is on the properties around $z_1$ and $z_2$, we only make discussion in the neighborhood of those two points.
 \begin{proposition}\label{f M}
     Let $f$ be the function given in \eqref{LG function}, then it is decreasing around $z_1$ and increasing around $z_2$. For any $x\le z_1$ such that $f''(x)$ exists, we have $f''(x)>0$.
 \end{proposition}
 \begin{proof} Directly taking the derivative of $f$ gives us:
 \begin{equation*}
     f'(x)=\frac{z_1+z_2}{4x^2}-\frac{z_1z_2}{2x^3}=\frac{(z_1+z_2)x-2z_1z_2}{4x^3}
 \end{equation*}
  and hence,
  \begin{equation*}
      f'(z_1)=\frac{z_1^2-z_1z_2}{4z_1^3}<0,~~f'(z_2)=\frac{z_2^2-z_1z_2}{4z_2^2}>0.
  \end{equation*}
  Obviously, $f'(x)<0$ for any $x\in(0,z_1)$ and $f'(y)>0$ for any $y\in(z_2,\infty)$ from the expression of $f'(x)$.
  The further derivative of $f'$ takes the form:
\begin{equation*}
      f''(x)=\frac{3z_1z_2-(z_1+z_2)x}{2x^4}.
\end{equation*}
  For any $x\le z_1$, the numerator above satisfies $3z_1z_2-(z_1+z_2)x\ge z_1(2z_2-z_1)>0$. Since the denominator is always positive, we get the required result. Notice we don't have the similar result around $z_2$ as the sign of $z_2(2z_1-z_2)$ is undetermined.\end{proof}

  \begin{proposition}\label{B4z}
      For $N$ large enough, we can find an $s_1>0$ such that for any $s$ with $s_1\le s\le \frac{2\tilde{\mu}_{n,N}}{3\tilde{\sigma}_{N,n}}$ and $z$ around $z_1$,
      \begin{equation*}
          \frac{2\kappa_N}{3}\zeta^\frac{3}{2}(z),~~\frac{2\kappa_{n-1,N}}{3}\zeta^\frac{3}{2}\left(\frac{\tilde{x}_N(s)}{\kappa_{n-1,N}}\right)~~and~~\frac{2\kappa_{n,N-1}}{3}\zeta^\frac{3}{2}\left(\frac{\tilde{x}_N(s)}{\kappa_{n,N-1}}\right)\ge s
      \end{equation*}
      where $\tilde{x}_N(s)=\tilde{\mu}_{n,N}-\tilde{\sigma}_{n,N}s$ with respect to $\tilde{\mu}_{n,N}$ and $\tilde{\sigma}_{n,N}$ restricted as \eqref{restriction of recaling}. In particular, the inequalities hold for $s_1\le s\le \frac{e_0\tilde{\mu}_{n,N}}{\tilde{\sigma}_{N,n}}$ since $e_0<\frac{2}{3}$.
  \end{proposition}
  \begin{proof} We first claim there exists an $s'_1$ such that $f^{\frac{1}{2}}(z)\ge\frac{2}{\sigma_{n,N}}$ for any $s$ with $\frac{2\tilde{\mu}_{n,N}}{3\tilde{\sigma}_{n,N}}\ge s\ge s'_1$ around $z_1$ when $N$ is large. In fact, once the claim is true, then
  \begin{equation*}
  \begin{split}
      \frac{2\kappa_N}{3}\zeta^{\frac{3}{2}}(z)&=\kappa_N\int_{z}^{z_1}f^{\frac{1}{2}}(t)dt=\kappa_N\int_{z_1-\frac{\sigma_{n,N}}{\kappa_N}s}^{z_1-\frac{\sigma_{n,N}}{\kappa_N}s'_1}f^{\frac{1}{2}}(t)dt+\kappa_N\int_{z_1-\frac{\sigma_{n,N}}{\kappa_N}s'_1}^{z_1}f^{\frac{1}{2}}(t)dt\\
      &\ge\kappa_N\int_{z_1-\frac{\sigma_{n,N}}{\kappa_N}s}^{z_1-\frac{\sigma_{n,N}}{\kappa_N}s'_1}f^{\frac{1}{2}}(t)dt\ge2(s-s'_1)\ge s
      \end{split}
  \end{equation*}
for $s\ge2s'_1$. The range of the integral above causes no problem, since with the order as in \eqref{restriction of recaling}, $z_1-\frac{\sigma_{n,N}}{\kappa_N}s\ge z_1-\frac{2\mu_{n,N}}{3\kappa_N}=\frac{z_1}{3}>0$ given $N$ is large.

Now we prove the claim. Since $z_1-\frac{\sigma_{n,N}}{\kappa_N}s>0$ from previous discussion, performing the Taylor expansion of $f$ around $z_1$ leads to: $f(z_1-\frac{\sigma_{n,N}}{\kappa_N}s)=f(z_1)+f'(z_1)(-\frac{\sigma_{n,N}}{\kappa_N}s)+\frac{f''(\theta)}{2}(\frac{\sigma_{n,N}}{\kappa_N}s)^2$ where $\theta$ is between $z_1-\frac{\sigma_{n,N}}{\kappa_N}s$ and $z_1$ and hence, positive. Therefore, $f''(\theta)>0$ according to Proposition \ref{f M}.
Using the fact that $f(z_1)=0$,
\begin{equation*}
    f\left(z_1-\frac{\sigma_{n,N}}{\kappa_N}s\right)\ge-f'(z_1)\frac{\sigma_{n,N}s}{\kappa_N}=\frac{z_2-z_1}{4z_1^2}\frac{\sigma_{n,N}}{\kappa_N}s.
\end{equation*}
To guarantee $f(z)\ge\frac{4}{\sigma^2_{n,N}}$, we only need to show $s\ge\frac{16\kappa_N}{\sigma^3_{n,N}}\frac{z_1^2}{z_2-z_1}$. Recalling that both two fractions on the right side have finite positive limits by our assumption that $\gamma\in(0,1)$, there exists $M$ such that for any $N\ge M$, the right side of the inequality above is bounded above by some positive constant $s'_1$ depending on $\gamma$. That is, for any $s\ge s'_1$, when $N\ge M$, we always have $s\ge\frac{16\kappa_N}{\sigma_{n,N}^3}\frac{z_1^2}{z_2-z_1}$. Recalling $\frac{2\tilde{\mu}_{n,N}}{3\tilde{\sigma}_{n,N}}\rightarrow\infty$ as $N\rightarrow\infty$, the inequality $s_1'<\frac{2\tilde{\mu}_{n,N}}{3\tilde{\sigma}_{n,N}}$ always holds for $N$ large.

Now let's consider the scaling $\tilde{x}_N(s)$. Since
\begin{equation*}
    \frac{2\kappa_{n-1,N}}{3}\zeta\left(\frac{\tilde{x}_N(s)}{\kappa_{n-1,N}}\right)=\kappa_{n-1,N}\int_{\frac{\tilde{x}_N(s)}{\kappa_{n-1,N}}}^{z_1}f^\frac{1}{2}(t)dt=\kappa_{n-1,N}\int_{z_1-\tilde{\epsilon}_{n-1,N}(s)}^{z_1}f^\frac{1}{2}(t)dt
\end{equation*}
and $\tilde{\epsilon}_{n-1,N}(t)-\tilde{\epsilon}_{n-1,N}(s)=\frac{\tilde{\sigma}_{n,N}}{\kappa_{n-1,N}}(t-s)$, we can similarly begin with the proof of $f^\frac{1}{2}(z)\ge\frac{2}{\tilde{\sigma}_{n,N}}$ for $s\in\left[s_1'',\frac{\tilde{2\mu}_{n,N}}{3\tilde{\sigma}_{n,N}}\right]$ and some $s_1''$. Noticing that
\begin{equation*}
    z_1-\tilde{\epsilon}_{n-1,N}(s)=\frac{1}{\kappa_{n-1,N}}\left(\tilde{\mu}_{n,N}-\tilde{\sigma}_{n,N}s\right)\ge\frac{1}{\kappa_{n-1,N}}\left(\tilde{\mu}_{n,N}-\frac{2\mu_{n-1,N}\tilde{\sigma}_{n,N}}{3\sigma_{n-1,N}}\right)>0,
\end{equation*}
our reasoning regarding $f''$ still works and hence,
\begin{equation*}
    f\left(z_1-\tilde{\epsilon}_{n-1,N}(s)\right)\ge\frac{z_2-z_1}{4z_1^2}\tilde{\epsilon}_{n-1,N}(s).
\end{equation*}
Consequently, to guarantee $f(z)\ge\frac{4}{\tilde{\sigma}^2_{n,N}}$, we only need
\begin{equation*}
    s\ge\frac{16\kappa_{n-1,N}}{\tilde{\sigma}^3_{n,N}}\frac{z_1^2}{z_2-z_1}+\frac{\tilde{\mu}_{n,N}-\mu_{n-1,N}}{\sigma_{n-1,N}}.
\end{equation*}
Again, similar to the case of $s_1'$, the positive limit of the right side indicates the existence of $s_1''$. Now we have
\begin{equation*}
\begin{split}
       \frac{2\kappa_{n-1,N}}{3}\zeta\left(\frac{\tilde{x}_N(s)}{\kappa_{n-1,N}}\right)\ge\kappa_{n-1,N}\int_{z_1-\tilde{\epsilon}_{n-1,N}(s)}^{z_1-\tilde{\epsilon}_{n-1,N}(s_1'')}f^\frac{1}{2}(t)dt\ge2(s-s_1'')\ge s
\end{split}
\end{equation*}
for $s\ge 2s_1''$. Similar computation for $\frac{2\kappa_{n,N-1}}{3}\zeta^\frac{3}{2}\left(\frac{\tilde{x}_N(s)}{\kappa_{n,N-1}}\right)$ leads to 
\begin{equation*}
    \frac{2\kappa_{n,N-1}}{3}\zeta^\frac{3}{2}\left(\frac{\tilde{x}_N(s)}{\kappa_{n,N-1}}\right)\ge 2(s-s_1''')\ge s
\end{equation*}
for some $s_1'''$. Define $s_1:=\max\left\{2s_1',2s_1'',2s_1'''\right\}$, then for any $s\in\left[s_1,\frac{2\tilde{\mu}_{n,N}}{3\tilde{\sigma}_{n,N}}\right]$, the required inequalities hold.
\end{proof}

\subsection{Proof of Lemma \ref{c and r}}\label{P4CandR}
Recall the L--G transform around $z_1$ is given in \eqref{LG1}. We first compute the explicit antiderivative of $f^\frac{1}{2}(t)=\frac{\sqrt{t^2-4t+\omega_N^2}}{2t}$ around $z_1$. It suffices to focus on $t\in(0,z_1]$ since the discussion is similar for $t>z_1$. Noticing that $\omega_N^2<4$ and $t^2-4t+\omega_N^2>0$ in our settings,
\begin{equation}\label{antiD1}
\begin{split}
       \int \frac{\sqrt{t^2-4t+\omega_N^2}}{2t}dt&=\frac{1}{4}\int\frac{2t-4}{\sqrt{t^2-4t+\omega_N^2}}dt-\int\frac{dt}{\sqrt{t^2-4t+\omega_N^2}}+\frac{\omega_N^2}{2}\int\frac{dt}{t\sqrt{t^2-4t+\omega_N^2}}\\
       &:=\frac{1}{2}\sqrt{t^2-4t+\omega_N^2}-I_1+\frac{\omega_N^2}{2}I_2.
\end{split}
\end{equation}
Using the fact that $\int\frac{1}{\sqrt{x^2-1}}dx=\log|x+\sqrt{x^2-1}|+C$ for $|x|>1$ and $C\in\mathbb{R}$, we see
\begin{equation*}
    \begin{split}
        \int\frac{dx}{\sqrt{(x-a)^2-b^2}}=\log|\sqrt{(x-a)^2-b^2}+(x-a)|+C
    \end{split}
\end{equation*}
for any $a,~b\in\mathbb{R}$ with $|x-a|>|b|$, and hence,
\begin{equation}\label{anti4D2}
    \begin{split}
        I_1=\log|\sqrt{t^2-4t+\omega_N^2}+(t-2)|+C.
    \end{split}
\end{equation}
Substituting $u$ by $\frac{1}{t}$ in $I_2$, since $dt=-\frac{1}{u^2}du$ and $t>0$ in our settings,
\begin{equation}\label{anti4D3}
    \begin{split}
        I_2=\frac{1}{\omega_N}\log t-\frac{1}{\omega_N}\log\left|\sqrt{\omega_N^2t^2-4\omega_N^2t+\omega_N^4}+\omega_N^2-2t\right|+C.
    \end{split}
\end{equation}
Plugging \eqref{anti4D2} and \eqref{anti4D3} into \eqref{antiD1}, the explicit antiderivative is given as:
\begin{equation}\label{anti4zLG}
    \begin{split}
        \int f^\frac{1}{4}(t)dt&=\frac{1}{2}\sqrt{t^2-4t+\omega_N^2}-\log|\sqrt{t^2-4t+\omega_N^2}+(t-2)|\\&~~+\frac{\omega_N}{2}\log t-\frac{\omega_N}{2}\log\left|\sqrt{\omega_N^2t^2-4\omega_N^2t+\omega_N^4}+\omega_N^2-2t\right|+C.
    \end{split}
\end{equation}
Since $z_1^2-4z_1+\omega_N^2=0$, performing the integral of \eqref{anti4zLG} from $z\in(0,z_1)$ to $z_1$ yields:
\begin{equation}\label{ex4zeta1}
    \begin{split}
        \frac{2}{3}\zeta^\frac{3}{2}(z)&=\int_z^{z_1}f^\frac{1}{2}(t)dt=\frac{\omega_N}{2}\log z_1-\log|z_1-2|-\frac{\omega_N}{2}\log|\omega_N^2-2z_1|\\
        &~~-\frac{1}{2}\sqrt{z^2-4z+\omega_N^2}+\log\left|\sqrt{z^2-4z+\omega_N^2}+z-2\right|-\frac{\omega_N}{2}\log z\\
        &~~+\frac{\omega_N}{2}\log\left|\sqrt{\omega_N^2z^2-4\omega_N^2z+\omega_N^4}+\omega_N^2-2z\right|\\&:=C_0(z)-\frac{a}{2\kappa_N}\log z.
    \end{split}
\end{equation}
In particular, define:
\begin{equation}\label{c_0}
    \begin{split}
       c_0&:=C_0(0)=\frac{\omega_N}{2}\log z_1-\log|z_1-2|-\frac{\omega_N}{2}\log|\omega_N^2-2z_1|-\frac{\omega_N}{2}\\&~~+\log|\omega_N-2|+\frac{\omega_N}{2}\log\left(2\omega_N^2\right)\\
       &=-\frac{\omega_N}{2}\log\left|z_2-2\right|+\log\left|\frac{\omega_N-2}{2-z_1}\right|-\frac{\omega_N}{2}+\frac{\omega_N}{2}\log2+\omega_N\log a-\omega_N\log\kappa_N\\&=\frac{a}{\kappa_N}\log a-\frac{a}{2\kappa_N}-\frac{a}{2\kappa_N}\log\kappa_N-\frac{a}{4\kappa_N}\log\left(N+\frac{1}{2}\right)\left(n+\frac{1}{2}\right)+\frac{1}{2}\log\left|\frac{N+\frac{1}{2}}{n+\frac{1}{2}}\right|,
    \end{split}
\end{equation}
we see $\frac{2}{3}\zeta^\frac{3}{2}\sim-\frac{a}{2\kappa_N}\log z\rightarrow\infty$ when $z\rightarrow0$ since $c_0$ is a finite number. Recall $w_N(x)=(x)^\frac{a+1}{2}\exp\left(-\frac{x}{2}\right)L_N^a(x)\sim x^\frac{a+1}{2}\frac{(N+a)!}{a!N!}$ when $x\rightarrow 0$, according to \eqref{w2} and \eqref{c_0},
\begin{equation}\label{conCN}
    \begin{split}
        c^{\{1\}}_N&=\lim_{z\rightarrow0}\frac{w_N(\kappa z)}{w_2(\kappa_N,z)}=\frac{(N+a)!}{a!N!}\lim_{z\rightarrow0}\frac{f^\frac{1}{4}(z)(\kappa_Nz)^\frac{a+1}{2}}{\zeta^\frac{1}{4}(z)Ai\left(\kappa_N^\frac{2}{3}\zeta\right)}\\
        &=\frac{2\sqrt{\pi}(N+a)!\kappa_N^\frac{3a+4}{6}}{a!N!}\lim_{z\rightarrow0}\frac{(z^2-4z+\omega_N^2)^\frac{1}{4}z^\frac{a+1}{2}\exp\left(\kappa_NC_0(z)\right)}{\sqrt{2z}z^{\frac{a}{2}}}\\
        &=\frac{\sqrt{2\pi a}(N+a)!\kappa_N^\frac{1}{6}}{a!N!}a^ae^{-\frac{a}{2}}\left(N+\frac{1}{2}\right)^{\frac{N+\frac{1}{2}}{2}}\left(n+\frac{1}{2}\right)^{-\frac{n+\frac{1}{2}}{2}}
    \end{split}
\end{equation}
where the large $x$ expansion: $Ai(x)\sim\frac{\exp\left(-\frac{2}{3}x^\frac{3}{2}\right)}{2\sqrt\pi x^\frac{1}{4}}$ has been used in the second line. Since $N!=\sqrt{2\pi}N^{N+\frac{1}{2}}e^{-N}(1+O(N^{-1}))$ for $N\in\mathbb{N}$ large by Stirling's formula, \eqref{conCN} further yields:
\begin{equation}\label{ConRN}
\begin{split}
       \left(r^{\{1\}}_N\right)^2&=\left(\frac{c^{\{1\}}_N}{\kappa_N^\frac{1}{6}}\right)^2\frac{N!}{(N+a)!}=\frac{2\pi a^{2a+1}}{e^{a}}\frac{\left(N+\frac{1}{2}\right)^{N+\frac{1}{2}}}{\left(n+\frac{1}{2}\right)^{n+\frac{1}{2}}}\frac{\sqrt{2\pi}n^{n+\frac{1}{2}}}{\sqrt{2\pi}N^{N+\frac{1}{2}}}\frac{e^{N-n+2a}}{2\pi a^{2a+1}}\left(1+O(N^{-1})\right)\\
       &=\frac{\left(1+\frac{1}{2N}\right)^{N+\frac{1}{2}}}{\left(1+\frac{1}{2n}\right)^{n+\frac{1}{2}}}(1+O(N^{-1}))=1+O(N^{-1})
\end{split}
\end{equation}
since $\frac{a}{N+a}\rightarrow1-\gamma$ and $\frac{N}{N+a}\rightarrow\gamma\in(0,1)$ by our assumption. Taking the square root completes the proof as $r^{\{1\}}_N\ge0$ by our construction. \qed

\subsection{Expansion for L--G Transform}\label{P4seq}
\begin{lemma}\label{seq}
    If $z=z_1-\epsilon$ for some suitable $\epsilon$ (say less or equal to $\frac{\mu_{n,N}}{2\kappa_{N}}$), then there exists a sequence $s_N$ with finite limit such that
    \begin{equation}\label{a4zeta}
    \zeta^{\frac{3}{2}}(z)=\frac{\sqrt{\epsilon^3(z_{2}-z_{1})}}{2z_{1}}\left(1+\frac{3}{5}s_{N}\epsilon+O(\epsilon^2)\right)
\end{equation}
and hence,
\begin{equation}\label{a4zeta2}
   \kappa_{N}^{\frac{2}{3}}\zeta(z)=\frac{\epsilon\kappa_{N}}{\sigma_{n,N}}\left(1+\frac{2\epsilon s_{N}}{5}+O(\epsilon^2)\right),
\end{equation}
\begin{equation}\label{a4inverse}
    \left(\frac{\kappa_{N}}{\sigma^3_{n,N}}\right)^{\frac{1}{6}}\tilde{f}^{-\frac{1}{4}}(z)=1-\frac{2}{5} s_{N}\epsilon+O(\epsilon^2).
\end{equation}
Here $z_1$, $z_2$, $\kappa$ and $s$ are all double-indexed sequences with respect to $(j,k)\in\mathbb{N}^2$. We will abbreviate them as in this lemma if $(j,k)=(n,N)$. Moreover, the L--G transform $\zeta$ and induced $\tilde{f}$ is always assumed to be around $z_1$ (with parameter $(n,N)$ here).
\end{lemma}
\begin{proof}
   According to the definition, 
\begin{equation*}
        \frac{2}{3}\zeta^{\frac{3}{2}}(z)=\frac{2}{3}\zeta^{\frac{3}{2}}\left(\frac{x_N(s)}{\kappa_N}\right)=\int_{z_1-\epsilon}^{z_1}\frac{\sqrt{(t-z_1)(t-z_2)}}{2t}dt.
\end{equation*}
Recall $\frac{\kappa_N}{\sigma_{n,N}^3}=\frac{z_2-z_1}{4z_1^2}$ by our construction. Define $A_N:=z_2-z_1=2\sqrt{4-\omega_N^2}$ and substitute $t$ by $y:=\frac{z_1-t}{\epsilon}$, we have:
\begin{equation*}
          \frac{2}{3}\zeta^{\frac{3}{2}}(z)=\frac{\sqrt{\epsilon^3A_N}}{2z_1}\int_0^1\sqrt y\frac{\sqrt{1+\frac{\epsilon}{A_N}y}}{1-\frac{\epsilon}{z_1}y}dy.
\end{equation*}
Since both $A_N$ and $z_1$ have finite nonzero limits as $\gamma\in(0,1)$, $\frac{\epsilon y}{A_N}$ and $\frac{\epsilon y}{z_1}$ are both small for any $y\in[0,1]$ provided $\epsilon$ in our region. The dominated convergence theorem allows us to perform the Taylor expansion inside the integral as:
\begin{equation}\label{int4epsilon}
\begin{split}
    \int_0^1\sqrt y\frac{\sqrt{1+\frac{\epsilon}{A_N}y}}{1-\frac{\epsilon}{z_1}y}dt&=\int_0^1\sqrt{y}\left(1+\epsilon y(\frac{1}{z_1}+\frac{1}{2A_N})+O(\epsilon^2)\right)dy\\
    &=\frac{2}{3}+\frac{2\epsilon}{5}\left(\frac{1}{z_1}+\frac{1}{2A_N}\right)+O(\epsilon^2).
\end{split}
\end{equation}
In fact, since
\begin{equation*}
    \lim_{N\rightarrow\infty}z_1=2-\frac{4\sqrt\gamma}{1+\gamma}=\frac{2(1-\sqrt\gamma)^2}{1+\gamma}
\end{equation*}
and
\begin{equation*}
    \lim_{N\rightarrow\infty}\frac{\mu_{n,N}}{\kappa_N}=\lim_{N\rightarrow\infty}\frac{2(\sqrt n-\sqrt N)^2}{N+n+1}=\frac{2(1-\sqrt\gamma)^2}{1+\gamma},
\end{equation*}
 we know $\epsilon\le\frac{\mu_{n,N}}{2\kappa_N}<\frac{2z_1}{3}$ for $N$ large enough by our restriction for $\epsilon$ and hence, $1-\frac{\epsilon}{z_1}y\ge \frac{1}{3}$ for any $y\in[0,1]$. Therefore, for $N$ large, the integrand in \eqref{int4epsilon} is bounded by $3\sqrt y\sqrt{1+C(\gamma)y}$, which is in $L^1[0,1]$. That means the dominated convergence theorem applies.

Define $s_N:=\frac{1}{z_1}+\frac{1}{2A_N}$. We know $s_N$ has finite limit and
\begin{equation*}
    \zeta^{\frac{3}{2}}(z)=\frac{\sqrt{\epsilon^3A_N}}{2z_1}\left(1+\frac{3}{5}s_N\epsilon+O(\epsilon^2)\right),
\end{equation*}
which is exactly \eqref{a4zeta}.

Recall $\frac{A_N}{4z_1^2}=\zeta'(z_1)=\frac{\kappa_N}{\sigma_{n,N}^3}$ by our construction, from \eqref{a4zeta},
\begin{equation*}
    \begin{split}
        \kappa_N^{\frac{2}{3}}\zeta&=\kappa_N^{\frac{2}{3}}\epsilon\left(\frac{A_N}{4z_1^2}\right)^\frac{1}{3}\left(1+\frac{2s_N\epsilon}{5}+O(\epsilon^2)\right)\\
        &=\frac{\epsilon\kappa_N}{\sigma_{n,N}}\left(1+\frac{2s_N\epsilon}{5}+O(\epsilon^2)\right)
    \end{split}.
\end{equation*}
Now write $f$ as:
\begin{equation*}
    f(z_1-\epsilon)=\frac{\epsilon(\epsilon+(z_2-z_1))}{4z_1^2(1-\frac{\epsilon}{z_1})^2}=\frac{A_N\epsilon(1+\frac{\epsilon}{A_N})}{4z_1^2(1-\frac{\epsilon}{z_1})^2}.
\end{equation*}
Again, performing the Taylor expansion for small $\epsilon$ leads to
\begin{equation*}
    f^{-\frac{3}{2}}(z)=\left(\frac{A_N\epsilon}{4z_1^2}\right)^{-\frac{3}{2}}\left(1-3s_N\epsilon+O(\epsilon^2)\right).
\end{equation*}
Recall $\tilde{f}=\frac{f}{\zeta}$, combing the previous results yields:
\begin{equation*}
    \tilde{f}^{-\frac{1}{4}}(z)=\left(\zeta^{\frac{3}{2}}(z)f^{-\frac{3}{2}}(z)\right)^{\frac{1}{6}}=\left(\frac{4z_1^2}{A_N}\right)^{\frac{1}{6}}\left(1-\frac{2}{5}s_N\epsilon+O(\epsilon^2)\right)
\end{equation*}
and hence,
\begin{equation*}
    \left(\frac{\kappa_N}{\sigma^3_{n,N}}\right)^{\frac{1}{6}}\tilde{f}^{-\frac{1}{4}}(z)=\left(\frac{\kappa_N}{\sigma^3_{n,N}}\times\frac{4z_1^2}{A_N}\right)^{\frac{1}{6}}\left(1-\frac{2}{5}s_N\epsilon+O(\epsilon^2)\right)=1-\frac{2}{5}s_N\epsilon+O(\epsilon^2)
\end{equation*}
for $N$ large enough.\end{proof}

\subsection{Proof of Lemma \ref{keyI}}\label{P4I1}
    Let $s_1$ be as in Proposition \ref{B4z} and define $s_0':=\max\{s_1,1\}$. Now we split $\left[s_0,\frac{e_0\tilde\mu_{n,N}}{\tilde\sigma_{n,N}}\right]$ into $[s_0,s_0']$, $\left[s_0',N^{\frac{1}{6}}\right]$ and $\left[N^\frac{1}{6},\frac{e_0\tilde\mu_{n,N}}{\tilde\sigma_{n,N}}\right]$ for $N$ large enough.

    According to \eqref{w2} and \eqref{B4epsilon}, 
    \begin{equation}\label{spB}
    \begin{split}
         B_{n,N}(x_N(s))&\le r_N\left|\left(\frac{\kappa_N}{\sigma_{n,N}^3}\right)^{\frac{1}{6}}\tilde{f}^{-\frac{1}{4}}(z)
\left(\frac{\mu_{n,N}}{x_N(s)}\right)-1\right|\left|Ai\left(\kappa_N^{\frac{2}{3}}\zeta\right)\right|\\&~~+r_N\left(\frac{\kappa_N}{\sigma_{n,N}^3}\right)^{\frac{1}{6}}\tilde{f}^{-\frac{1}{4}}(z)
\left(\frac{\mu_{n,N}}{x_N(s)}\right)|\epsilon_2(\kappa_N,z)|.
    \end{split}
    \end{equation}
For any $s\in\left[s_0,\frac{e_0\tilde\mu_{n,N}}{\tilde\sigma_{n,N}}\right]$, since $\frac{e_0\tilde\mu_{n,N}}{\tilde\sigma_{n,N}}\le\frac{2\mu_{n,N}}{3\sigma_{n,N}}$ for large $N$ given the condition in \eqref{restriction of recaling}, Proposition \ref{ex1} and Proposition \ref{ex2} yield:
\begin{equation*}
    r_N\left(\frac{\kappa_N}{\sigma^3_{n,N}}\right)^\frac{1}{6}\tilde{f}^{-\frac{1}{4}}(z)\left(\frac{\mu_{n,N}}{x_N(s)}\right)|\epsilon_2(\kappa_N,z)|\le\frac{C(s_0)}{\kappa_N}e^{-\frac{s_0}{2}}
\end{equation*}
and hence, we only have to pay attention to the first term on the right of \eqref{spB} since $\kappa_N=\Theta\left(N\right)$.

    Recall $x_N(s)=\mu_{n,N}-\sigma_{n,N}s$, $\mu_{n,N}=\Theta(N)$ and $\sigma_{n,N}=\Theta\left(N^{\frac{1}{3}}\right)$. For any $s\in\left[s_0',N^{\frac{1}{6}}\right]$, since $\epsilon_N(s):=\frac{\sigma_{n,N}s}{\kappa_N}=O\left(N^{-\frac{1}{2}}\right)$, applying \eqref{a4inverse} implies:
\begin{equation}\label{seq1}
        \begin{split}
            \left(\frac{\kappa_N}{\sigma_{n,N}^3}\right)^{\frac{1}{6}}&\tilde{f}^{-\frac{1}{4}}(z)
\left(\frac{\mu_{n,N}}{x_N(s)}\right)=\left(1-\frac{2}{5}s_N\epsilon_N(s)+O\left(\epsilon_N^2(s)\right)\right)\frac{1}{1-\frac{\sigma_{n,N}s}{\mu_{n,N}}}\\
&=\left(1-\frac{2}{5}s_N\epsilon_N(s)+O\left(\epsilon_N^2(s)\right)\right)\left(1+\frac{\sigma_{n,N}s}{\mu_{n,N}}+O\left(\sigma^2_{n,N}s^2\mu_{n,N}^{-2}\right)\right)\\
&=1+\left(\frac{\sigma_{n,N}}{\mu_{n,N}}-\frac{2s_N\sigma_{n,N}}{5\kappa_N}\right)s+O\left(N^{-1}\right):=1+s_N's+O\left(N^{-1}\right)
        \end{split}
    \end{equation}
on this interval. Since $s_N'=\Theta\left(N^{-\frac{2}{3}}\right)$, we conclude that there exists a constant $C(\gamma)$ such that
\begin{equation*}
   N^{\frac{2}{3}} r_N\left|\left(\frac{\kappa_N}{\sigma_{n,N}^3}\right)^{\frac{1}{6}}\tilde{f}^{-\frac{1}{4}}(z)
\left(\frac{\mu_{n,N}}{x_N(s)}\right)-1\right|\le C(\gamma)s
\end{equation*}
for any $s\in\left[s_0',N^{\frac{1}{6}}\right]$ when $N$ large. As indicated in Proposition \ref{B4z}, for $N$ large, we have $\kappa_N^{\frac{2}{3}}\zeta(z)\ge \left(\frac{3}{2}s\right)^{\frac{2}{3}}\ge1$ and hence, by the bound of $Ai$ in \eqref{ExB4Airy},
     \begin{equation}\label{Aizeta}
         \left|Ai\left(\kappa_N^{\frac{2}{3}}\zeta\right)\right|\le C(\gamma)e^{-s}
     \end{equation}
     for some constant $C(\gamma)$ and any $s\in\left[s_0',N^{\frac{1}{6}}\right]$ by \eqref{ExB4Airy}. Combining previous results, when $N$ is large, the first term of $B_{n,N}$ is bounded as:
     \begin{equation}\label{1st4B}
         N^{\frac{2}{3}} r_N\left|\left(\frac{\kappa_N}{\sigma_{n,N}^3}\right)^{\frac{1}{6}}\tilde{f}^{-\frac{1}{4}}(z)
\left(\frac{\mu_{n,N}}{x_N(s)}\right)-1\right|\left|Ai\left(\kappa_N^{\frac{2}{3}}\zeta\right)\right|\le C(\gamma)(se^{-s})\le C(\gamma)e^{-\frac{s}{2}}
     \end{equation}
     for some constant $C(\gamma)$ and any $s\in\left[s_0',N^{\frac{1}{6}}\right]$. The inequality \eqref{Aizeta} simply indicates the required bound for $B_{n,N}$ in $\left[N^\frac{1}{6},\frac{e_0\tilde\mu_{n,N}}{\tilde\sigma_{n,N}}\right]$. As $s\ge N^\frac{1}{6}$ in this region, it follows that $N^\frac{2}{3}\le e^\frac{s}{2}$ and hence, $N^\frac{2}{3}e^{-s}\le e^{-\frac{s}{2}}$ when $N$ is large. Since $|Ai(x)|\le\left|\frac{\mathbf{M}(x)}{\mathbf{E}(x)}\right|$ in Proposition \ref{aux4AiandBi}, combining Proposition \ref{ex1} and \eqref{Aizeta} leads to:
     \begin{equation}\label{2nd4B}
         \begin{split}
             &N^\frac{2}{3}r_N\left|\left(\frac{\kappa_N}{\sigma_{n,N}^3}\right)^{\frac{1}{6}}\tilde{f}^{-\frac{1}{4}}(z)
\left(\frac{\mu_{n,N}}{x_N(s)}\right)-1\right|\left|Ai\left(\kappa_N^{\frac{2}{3}}\zeta\right)\right|\\
&\le N^\frac{2}{3}r_N\left|\left(\frac{\kappa_N}{\sigma_{n,N}^3}\right)^{\frac{1}{6}}\tilde{f}^{-\frac{1}{4}}(z)
\left(\frac{\mu_{n,N}}{x_N(s)}\right)Ai\left(\kappa_N^{\frac{2}{3}}\zeta\right)\right|+N^\frac{2}{3}r_N\left|Ai\left(\kappa_N^{\frac{2}{3}}\zeta\right)\right|\\
&\le N^\frac{2}{3}r_N\left|\left(\frac{\kappa_N}{\sigma_{n,N}^3}\right)^{\frac{1}{6}}\tilde{f}^{-\frac{1}{4}}(z)
\left(\frac{\mu_{n,N}}{x_N(s)}\right)\frac{\mathbf{M}\left(\kappa_N^{\frac{2}{3}}\zeta\right)}{\mathbf{E}\left(\kappa_N^{\frac{2}{3}}\zeta\right)}\right|+N^\frac{2}{3}r_N\left|Ai\left(\kappa_N^{\frac{2}{3}}\zeta\right)\right|\\
&\le N^\frac{2}{3}C(\gamma)e^{-s}\le C(\gamma)e^{-\frac{s}{2}}.
         \end{split}
     \end{equation}

     Now consider $s\in[s_0,s_0']$. As both $s_0$ and $s_0'$ are fixed finite numbers, $\epsilon_N(s)$ is very small when $N$ is large. Therefore, Lemma \ref{seq} still applies and similar argument leads to
     \begin{equation}\label{Tf4 fixed}
          N^{\frac{2}{3}} r_N\left|\left(\frac{\kappa_N}{\sigma_{n,N}^3}\right)^{\frac{1}{6}}\tilde{f}^{-\frac{1}{4}}(z)
\left(\frac{\mu_{n,N}}{x_N(s)}\right)-1\right|\le C(\gamma,s_0)\max\{|s|,1\}
     \end{equation}
     for any $s\in[s_0,s_0']$ and $N$ large. Recall $x_N(s)=\mu_{n,N}-\sigma_{n,N}s$, we know $x_N(s)>0$ for any $s\in[s_0,s_0']$ when $N$ is large. Hence, by the integral expression of $\zeta$, it is easy to see $\zeta(z)=\zeta\left(\frac{x_N(s)}{\kappa_N}\right)$ is an increasing function with respect to $s$ on the same interval. Combining two limits $\kappa_N^\frac{2}{3}\zeta\left(\frac{x_N(s_0)}{\kappa_N}\right)\rightarrow s_0$ and $\kappa_N^\frac{2}{3}\zeta\left(\frac{x_N(s'_0)}{\kappa_N}\right)\rightarrow s'_0$,
     \begin{equation*}
         -2|s_0|\le\kappa_N^{\frac{2}{3}}\zeta\left(\frac{x_N(s)}{\kappa_N}\right)\le2s_0'
     \end{equation*}
     for any $s\in[s_0,s_0']$ and large $N$. As a result, when $N$ is large, the analysis in \eqref{Tf4 fixed} implies:
     \begin{equation}\label{3rd4B}
         \begin{split}
             &N^{\frac{2}{3}} r_N\left|\left(\frac{\kappa_N}{\sigma_{n,N}^3}\right)^{\frac{1}{6}}\tilde{f}^{-\frac{1}{4}}(z)
\left(\frac{\mu_{n,N}}{x_N(s)}\right)-1\right|\left|Ai\left(\kappa_N^\frac{2}{3}\zeta\right)\right|\\
&\le C(\gamma,s_0)\max\{|s|,1\}\left|Ai\left(\kappa_N^\frac{2}{3}\zeta\right)\exp\left(\frac{s}{2}\right)\right|e^{-\frac{s}{2}}\\
&\le e^{-\frac{s}{2}}C(\gamma,s_0)\sup_{s\in[-2|s_0|,2|s_0'|]}|\max\{|s|,1\}Ai(s)e^{\frac{s}{2}}|\\&=C(\gamma,s_0)e^{-\frac{s}{2}}
         \end{split}
     \end{equation}
     for any $s\in[s_0,s_0']$. Combining \eqref{1st4B}, \eqref{2nd4B} and \eqref{3rd4B} altogether finishes the proof for $B_{n,N}$. 

     Next let's consider $B_{n-1,N}$. From the previous discussion, for $x_{n-1,N}(s)=\mu_{n-1,N}-\sigma_{n-1,N}s$, we already have $N^\frac{2}{3}B_{n-1,N}(x_{n-1,N}(x))\le C(s_0,\gamma)e^{-\frac{s}{2}}$. Now our aim is to show the inequality still holds if we replace $x_{n-1,N}(s)$ by $\tilde{x}_N(s)$. 
     
     Introduce notations similar to those in \cite{el2006rate}:
\begin{equation}\label{tildedeviation}
    \begin{array}{cc}
      \delta_{n-1,N}:=\frac{\mu_{n-1,N}-\tilde{\mu}_{n,N}}{\kappa_{n-1,N}};~~~
      \delta_{n,N-1}:=\frac{\mu_{n,N-1}-\tilde{\mu}_{n,N}}{\kappa_{n,N-1}};\\
      \\
      \tilde{\epsilon}_{n,N}(s):=\frac{\mu_{n,N}-\tilde{\mu}_{n,N}}{\kappa_{n,N}}+\frac{\tilde{\sigma}_{n,N}}{\kappa_{n,N}}s.
    \end{array}
\end{equation}
We first derive sequences as in Lemma \ref{P4seq} and then argue through the same philosophy for $B_{n,N}$. For any $s\in\left[s_0,\frac{\mu_{n-1,N}}{3\sigma_{n-1,N}}\right]$ and $N$ large, notice
\begin{equation*}
    \frac{\tilde{x}_N(s)}{\kappa_{n-1,N}}=\frac{\mu_{n-1,N}}{\kappa_{n-1,N}}-\left(\frac{\mu_{n-1,N}-\tilde{\mu}_{n,N}}{\kappa_{n-1,N}}+\frac{\tilde{\sigma}_{n,N}}{\kappa_{n-1,N}}s\right)=z_1(n-1,N)-\tilde{\epsilon}_{n-1,N}(s)
\end{equation*}
and
\begin{equation*}
    \tilde{\epsilon}_{n-1,N}(s)\le\frac{\mu_{n-1,N}-\tilde{\mu}_{n,N}}{\kappa_{n-1,N}}+\frac{\tilde{\sigma}_{n,N}}{3\sigma_{n-1,N}}\mu_{n-1,N}\le\frac{\mu_{n-1,N}}{2\sigma_{n-1,N}}.
\end{equation*}
Given $\tilde{\mu}_{n,N}$ and $\tilde{\sigma}_{n,N}$ restricted as in \eqref{restriction of recaling}, Lemma \ref{P4seq} applies. Thus, it follows that
\begin{equation}\label{seq'1}
    \begin{split}
        \kappa_{n-1,N}^\frac{2}{3}\zeta\left(\frac{\tilde{x}_N(s)}{\kappa_{n-1,N}}\right)&=\kappa_{n-1,N}^\frac{2}{3}\zeta\left(z_1-\tilde{\epsilon}_{n-1,N}\right)\\&=\frac{\tilde{\epsilon}_{n-1,N}\kappa_{n-1,N}}{\sigma_{n-1,N}}\left(1+\frac{2s_{n-1,N}}{5}\tilde{\epsilon}_{n-1,N}+O\left(\tilde{\epsilon}^2_{n-1,N}\right)\right)
    \end{split}
\end{equation}
and
\begin{equation}\label{seq'2}
        \left(\frac{\kappa_{n-1,N}}{\sigma^3_{n-1,N}}\right)^\frac{1}{6}\tilde{f}^{-\frac{1}{4}}\left(\frac{\tilde{x}_N(s)}{\kappa_{n-1,N}}\right)=1-\frac{2}{5}s_{n-1,N}\tilde{\epsilon}_{n-1,N}+O\left(\tilde{\epsilon}^2_{n-1,N}\right).
\end{equation}

     For $s\in[s_0,s_0']\cup\left[N^\frac{1}{6},\frac{e_0\tilde\mu_{n,N}}{\tilde\sigma_{n,N}}\right]$, our previous discussion trivially applies since we still have $\kappa_{n-1,N}^\frac{2}{3}\zeta\left(\frac{\tilde{x}_N(s_0)}{\kappa_N}\right)\rightarrow s_0$ and $\kappa_{n-1,N}^\frac{2}{3}\zeta\left(\frac{\tilde{x}_N(s'_0)}{\kappa_N}\right)\rightarrow s'_0$ on   
     $[s_0,s_0']$ and $\kappa_{n-1,N}^\frac{2}{3}\zeta\left(\frac{\tilde{x}_N(s)}{\kappa_{n-1,N}}\right)\ge s\ge1$ on $\left[N^\frac{1}{6},\frac{e_0\tilde\mu_{n,N}}{\tilde\sigma_{n,N}}\right]$ according to \eqref{seq'1}.
     Therefore, we only need to focus on the interval $\left[s_0',N^{\frac{1}{6}}\right]$. Recall $\tilde{\epsilon}_{n-1,N}(s):=\epsilon=O\left(N^{-\frac{1}{2}}\right)$, similar to \eqref{seq1}, \eqref{seq'2} yields
\begin{equation}\label{seq2}
\begin{split}
\left(\frac{\kappa_{n-1,N}}{\sigma_{n-1,N}^3}\right)^{\frac{1}{6}}\tilde{f}^{-\frac{1}{4}}(z)
\left(\frac{\mu_{n-1,N}}{\tilde{x}_N(s)}\right)&=\left(1-\frac{2}{5}s_{n-1,N}\epsilon+O\left(\epsilon^2\right)\right)\frac{\mu_{n-1,N}}{\mu_{n-1,N}-\kappa_{n-1,N}\epsilon}\\
&=1+\left(\frac{\tilde\sigma_{n,N}}{\mu_{n-1,N}}-\frac{2s_{n-1,N}\tilde\sigma_{n,N}}{5\kappa_{n-1,N}}\right)s+O\left(N^{-1}\right)\\&:=1+s_{n-1,N}'s+O\left(N^{-1}\right).
\end{split}
\end{equation}
Again, as $s_{n-1,N}'=\Theta\left(N^{-\frac{2}{3}}\right)$, we still have:
\begin{equation*}
   N^{\frac{2}{3}} r_{n-1,N}\left|\left(\frac{\kappa_{n-1,N}}{\sigma_{n-1,N}^3}\right)^{\frac{1}{6}}\tilde{f}^{-\frac{1}{4}}(z)
\left(\frac{\mu_{n-1,N}}{\tilde{x}_N(s)}\right)-1\right|\le C(\gamma)s
\end{equation*}
for any $s\in\left[s_0',N^{\frac{1}{6}}\right]$ and some constant $C(s_0,\gamma)$ when $N$ is large. The bound of $Ai\left(\kappa_{n-1,N}^\frac{2}{3}\zeta\right)$ similarly follows from \eqref{ExB4Airy} since $\kappa_{n-1,N}^\frac{2}{3}\zeta\ge s\ge 1$ according to Proposition \ref{B4z}.

So far, the proof regarding $B_{n-1,N}$ has been finished. The inequality for $B_{n,N-1}$ can be proved similarly.
\qed   

\subsection{Proof of Lemma \ref{keyI2}}\label{P4I2}
     The bound for $D_{n,N}$ is essentially the difference $\left|Ai\left(\kappa_N^\frac{2}{3}\zeta\right)-Ai(s)\right|$. Since $\kappa_N^\frac{2}{3}\zeta(x_N(s))\rightarrow s$ and $Ai(x)\in C^2(\mathbb{R})$ with $|Ai'(x)|$ decreasing on $\mathbb{R}^+$, we simply apply the mean value theorem to establish the bound for $D$.

     Now split $\left[s_0,\frac{e_0\tilde\mu_{n,N}}{\tilde\sigma_{n,N}}\right]$ into $[s_0,2s_0']$, $\left[2s_0',N^{\frac{1}{6}}\right]$ and $\left[N^\frac{1}{6},\frac{e_0\tilde\mu_{n,N}}{\tilde\sigma_{n,N}}\right]$ where $s_0'$ is defined as in Appendix \ref{P4I1}. For any $s\in[2s_0',N^{\frac{1}{6}}]$, it is easy to see that $\min\left\{\kappa_N^\frac{2}{3}\zeta,s\right\}\ge\frac{s}{2}\ge s_0'\ge1$ for $N$ large. Moreover, the expansion
     \begin{equation}\label{a4zeta3}
     \begin{split}
         \kappa_N^\frac{2}{3}\zeta(x_N(s))&=\frac{\epsilon_N(s)\kappa_N}{\sigma_{n,N}}\left(1+\frac{2}{5}\epsilon_N(s)s_N+O(\epsilon_N^2(s))\right)\\&=s+\frac{2\sigma_{n,N}s_N}{5\kappa_N}s^2+O\left(\frac{s^3\sigma_{n,N}^2}{\kappa_N^2}\right) 
     \end{split}
     \end{equation}
     follows from \eqref{a4zeta2} and $\epsilon_N(s)=O\left(N^{-\frac{1}{2}}\right)$. Using \eqref{a4zeta3} and the bound of $|Ai'|$ in \eqref{ExB4Airy'}, we conclude that for any $s\in\left[2s_0',N^{\frac{1}{6}}\right]$, there exists a $\delta$ between $\kappa_N^\frac{2}{3}\zeta$ and $s$ such that
     \begin{equation}\label{1st4D}
         \begin{split}
             \left|Ai\left(\kappa_N^\frac{2}{3}\zeta\right)-Ai(s)\right|&=|Ai'(\delta)|\left|\kappa_N^\frac{2}{3}\zeta-s\right|\le C(\gamma)\left|Ai'\left(\frac{s}{2}\right)\right|\frac{\sigma_{n,N}}{\kappa_N}s^2\\
             &\le\frac{C(\gamma)\sigma_{n,N}}{\mu_{n,N}}s^{\frac{9}{4}}\exp\left(-\frac{s^{\frac{3}{2}}}{3\sqrt2}\right)\\&\le C(\gamma)\frac{\sigma_{n,N}}{\mu_{n,N}}e^{-\frac{s}{2}}
         \end{split}
     \end{equation}
     for $N$ large.

     For $s\in[s_0,2s_0']\cup\left[N^\frac{1}{6},\frac{e_0\tilde\mu_{n,N}}{\tilde\sigma_{n,N}}\right]$, the reasoning is similar to it is for $B_{n,N}$. For $s\in[s_0,2s_0']$, when $N$ is large enough, $\kappa_N^\frac{2}{3}\zeta$ is bounded by $-2|s_0|$ and $4s_0'$. Adapting the first line of \eqref{1st4D} yields:
     \begin{equation}\label{2nd4D}
         \begin{split}
             N^{\frac{2}{3}}\left|Ai\left(\kappa_N^\frac{2}{3}\zeta\right)-Ai(s)\right|&\le N^\frac{2}{3}\left|\kappa_N^\frac{2}{3}\zeta-s\right|\sup_{-2|s_0|\le\delta\le4s_0'}|Ai'(\delta)|\le C(\gamma,s_0)s^2\\&\le e^{-\frac{s}{2}} C(\gamma,s_0)\sup_{s_0\le s\le2s_0'}s^2e^{\frac{s}{2}}\\&=C(\gamma,s_0)e^{-\frac{s}{2}}.
         \end{split}
     \end{equation}
     On $\left[N^\frac{1}{6},\frac{e_0\tilde\mu_{n,N}}{\tilde\sigma_{n,N}}\right]$, when $N$ is large enough, it follows that
     \begin{equation}\label{3rd4D}
         N^\frac{2}{3}\left|Ai\left(\kappa_N^\frac{2}{3}\zeta\right)-Ai(s)\right|\le N^\frac{2}{3}\left|Ai\left(\kappa_N^\frac{2}{3}\zeta\right)\right|+N^\frac{2}{3}\left|Ai(s)\right|\le C(\gamma)e^{-\frac{s}{2}}
     \end{equation}
     from \eqref{ExB4Airy} and \eqref{Aizeta}. Estimates \eqref{3rd4B}, \eqref{1st4D}) and \eqref{2nd4D} directly provide us with the bound of $D_{n,N}$ with the required order.\qed
     
\subsection{Proof of Lemma \ref{re4D}}\label{P4re4D}
For $s\in[s_0,2s_0']$, the expansion \eqref{seq'1} yields $\kappa_{n-1,N}^\frac{2}{3}\zeta\left(\frac{\tilde{x}_N(s_0)}{\kappa_{n-1N}}\right)\rightarrow s_0$ and $\kappa_{n-1,N}^\frac{2}{3}\zeta\left(\frac{\tilde{x}_N(s'_0)}{\kappa_{n-1N}}\right)\rightarrow s'_0$. Thus, the expansion \eqref{2nd4D} still holds if we substitute $x_{n-1,N}(s)$ by $\tilde{x}_N(s)$ in $D_{n-1,N}$. Similarly, for $s\in\left[N^\frac{1}{6},\frac{e_0\tilde\mu_{n,N}}{\tilde\sigma_{n,N}}\right]$, since $\kappa_{n-1,N}^\frac{2}{3}\left(\frac{\tilde{x}_N(s)}{\kappa_{n-1,N}}\right)\ge s$ by Proposition \ref{B4z}, the inequality \eqref{3rd4D} also holds for $D_{n-1,N}(\tilde{x}_N(s))$. Similar discussion also makes sense for $D_{n,N-1}(\tilde{x}_N(s))$. Therefore, we can restrict our focus to $\left|Ai\left(\kappa_N^\frac{2}{3}\zeta\right)-Ai(s)\right|$ when $s\in\left[2s_0',N^\frac{1}{6}\right]$. Further consider the differences:
\begin{equation*}
    \begin{array}{ll}
        d_{n-1,N}(s):=\kappa_{n-1,N}^\frac{2}{3}\zeta(\tilde{x}_N(s))-s,\\
        \\
        d_{n,N-1}(s):=\kappa_{n,N-1}^\frac{2}{3}\zeta(\tilde{x}_N(s))-s.
    \end{array}
\end{equation*}
Performing the Taylor expansion around $s$ for $Ai$ with the integral remainder yields:
\begin{equation}\label{Taylor4Ai}
    \begin{split}
         Ai\left(\kappa_{n-1,N}^\frac{2}{3}\zeta(\tilde{x}_N(s))\right)&=Ai(s)+d_{n-1,N}(s)Ai'(s)+\int_0^{d_{n-1,N}(s)}(d_{n-1,N}(s)-t)(s+t)Ai(s+t)dt\\
        &:=Ai(s)+d_{n-1,N}(s)Ai'(s)+R_{n-1,N}(s)
    \end{split}
\end{equation}
where the Airy equation: $Ai''(y)=yAi(y)$ has been used inside the integral. Consequently,
\begin{equation}\label{B4D}
\begin{split}
      D(\tilde{x}_N(s))&\le|\theta_{n-1,N}d_{n-1,N}(s)+\theta_{n,N-1}d_{n,N-1}(s)||Ai'(s)|\\
      &~~+\theta_{n-1,N}|R_{n-1,N}(s)|+\theta_{n,N-1}|R_{n,N-1}(s)|.
\end{split}
\end{equation}
The expansion \eqref{a4zeta2} indicates:
\begin{equation}\label{a4D}
    \begin{split}
        d_{n-1,N}(s)&=\frac{\kappa_{n-1,N}\tilde{\epsilon}_{n-1,N}}{\sigma_{n-1,N}}\left(1+\frac{2s_{n-1,N}\tilde{\epsilon}_{n-1,N}}{5}+O\left(\tilde{\epsilon}_{n-1,N}^2\right)\right)-s\\
        &=\frac{\kappa_{n-1,N}}{\sigma_{n-1,N}}\delta_{n-1,N}+\left(\frac{\tilde{\sigma}_{n,N}}{\sigma_{n-1,N}}-1\right)s+\frac{2s_{n-1,N}\kappa_{n-1,N}}{5\sigma_{n-1,N}}\tilde{\epsilon}^2_{n-1,N}+O\left(\frac{\tilde{\epsilon}^3_{n-1,N}\kappa_{n-1,N}}{\sigma_{n-1,N}}\right)\\
        &:=\frac{\kappa_{n-1,N}}{\sigma_{n-1,N}}\delta_{n-1,N}+\left(\frac{\tilde{\sigma}_{n,N}}{\sigma_{n-1,N}}-1\right)s+R'_{n-1,N}(s).
    \end{split}
\end{equation}
and similarly,
\begin{equation*}
    d_{n,N-1}(s):=\frac{\kappa_{n,N-1}}{\sigma_{n,N-1}}\delta_{n,N-1}+\left(\frac{\tilde{\sigma}_{n,N}}{\sigma_{n,N-1}}-1\right)s+R'_{n,N-1}(s).
\end{equation*}
Further defining
\begin{equation*}
    U_N:=\theta_{n-1,N}\frac{\kappa_{n-1,N}}{\sigma_{n-1,N}}\delta_{n-1,N}+\theta_{n,N-1}\frac{\kappa_{n,N-1}}{\sigma_{n,N-1}}\delta_{n,N-1}
\end{equation*}
and
\begin{equation*}
    V_N:=\theta_{n-1,N}\left(\frac{\tilde{\sigma}_{n,N}}{\sigma_{n-1,N}}-1\right)+\theta_{n,N-1}\left(\frac{\tilde{\sigma}_{n,N}}{\sigma_{n,N-1}}-1\right),
\end{equation*}
we can write \eqref{B4D} as:
\begin{equation*}
\begin{split}
      D(\tilde{x}_N(s))&\le
    \theta_{n-1,N}\left(|R_{n-1,N}(s)|+|Ai'(s)R'_{n-1,N}(s)|\right)\\&+\theta_{n,N-1}\left(|R_{n,N-1}(s)|+|Ai'(s)R'_{n,N-1}(s)|\right)+|U_N+V_Ns||Ai'(s)|.
\end{split}
\end{equation*}
As $\theta_{n-1,N}-1$ and $\theta_{n,N-1}-1$ are both of order $N^{-1}$, according to Proposition \ref{B4re} and Proposition \ref{error4d}, the remainders $R$ and $R'$ do not cause any problem for our rate. Through the same core idea as in \cite{el2006rate}, to achieve the promised rate of convergence, we only need to pick $\tilde{\mu}_{n,N}$ and $\tilde{\sigma}_{n,N}$ such that $U_N=V_N=0$ to totally offset the effect of $|Ai'(s)||U_N+V_Ns|$ in $D(\tilde{x}_N(s))$. Recalling that $\theta_{n-1,N}=\sqrt{Nn}\sigma_{n-1,N}^{1/2}\frac{\tilde{\sigma}_{n,N}}{\tilde{\mu}_{n,N}}$ and $\theta_{n,N-1}=\sqrt{Nn}\sigma_{n,N-1}^{1/2}\frac{\tilde{\sigma}_{n,N}}{\tilde{\mu}_{n,N}}$, we can further express $U_N$ and $U_N$ as:
\begin{equation*}
\begin{array}{cc}
      U_N=\tilde{\sigma}_{n,N}\sqrt{Nn}\left(\left(\frac{\tilde{\mu}_{n,N}}{\mu_{n-1,N}}-1\right)\frac{1}{\sigma_{n-1,N}^{1/2}}+\left(\frac{\tilde{\mu}_{n,N}}{\mu_{n,N-1}}-1\right)\frac{1}{\sigma_{n,N-1}^{1/2}}\right),\\
      \\
        V_N=\tilde{\sigma}_{n,N}\sqrt{Nn}\left(\left(\frac{\tilde{\sigma}_{n,N}}{\sigma_{n-1,N}}-1\right)\frac{\sigma_{n-1,N}^{1/2}}{\mu_{n-1,N}}+\left(\frac{\tilde{\sigma}_{n,N}}{\sigma_{n,N-1}}-1\right)\frac{\sigma_{n,N-1}^{1/2}}{\mu_{n,N-1}}\right).
\end{array}
\end{equation*}
Hence, picking $\tilde{\mu}_{n,N}$ and $\tilde{\sigma}_{n,N}$ as in \eqref{recenter} and \eqref{rescal} helps us totally kill $U_N$ and $V_N$.\qed

\subsection{Proof of Lemma \ref{inlarge}}\label{P4inlarge}
    We use the same $C(\gamma,r)$ to denote different constants in the proof below to simplify notation.

   Consider $G_{\tau}$ and $H_{\tau}$. In fact, noticing that $\tilde{\mu}_{n,N}-\tilde{\sigma}_{n,N}t\le0$ for any $t\in\left(\frac{\tilde{\mu}_{n,N}}{\tilde{\sigma}_{n,N}},+\infty\right)$ and hence, $\xi_{\tau'}(\tilde{x}_N(t))=\eta_{\tau'}(\tilde{x}_N(t))=0$ when $\tilde{x}_N:=\tilde{\mu}_{n,N}-\tilde{\sigma}_{n,N}t$ in this regime, we only draw attention to $\left[\frac{e_0\tilde\mu_{n,N}}{\tilde\sigma_{n,N}},\frac{\tilde{\mu}_{N,n}}{\tilde{\sigma}_{N,n}}\right]$. Recall $N$ is assumed to be odd,
\begin{equation*}
\begin{split}
    \xi_{\tau}(t)&=\xi_{\tau}(\tilde{x}_N(t))=\frac{-1}{\sqrt2}\left(\frac{\sqrt{Nn}\sigma_{n-1,N}^\frac{1}{2}\tilde{\sigma}_{n,N}}{\mu_{n-1,N}}\right)F_{n-1,N}\left(\tilde{x}_N\right)\left(\frac{\mu_{n-1,N}}{\tilde{x}_N}\right)\\
    &=\frac{-1}{\sqrt2}\left(\frac{\sigma^{\frac{1}{2}}_{n-1,N}\tilde{\sigma}_{n,N}}{\sigma_{n,N}^{\frac{1}{2}}}\right)\sqrt{\frac{Nn(N)!}{(n-1)!}}\tilde{x}_n^{\frac{n-N-2}{2}}L^{n-N-1}_N(\tilde{x}_N)e^{-\frac{\tilde{x}_N}{2}}
    \end{split}
\end{equation*}
The fact that $\xi_{\tau}=0$ for $\tilde{x}_N<0$ is obvious as argued above. Furthermore, as $n-N>2$ when $N$ is large by our assumption, we see the function $\xi_{\tau}$ is well-defined for any $\tilde{x}_N\ge 0$ (the power of $\tilde{x}_N$ is always positive). Recall  $1-e_0=\frac{1}{e}$, applying Lemma \ref{B4LagPoly} to $\xi_{\tau}$,
\begin{equation}\label{varsmall}
\begin{split}
       |\xi_{\tau'}(t)|&\le\frac{1}{\sqrt2}\left(\frac{\sigma^{\frac{1}{2}}_{n-1,N}\tilde{\sigma}_{n,N}}{\sigma_{n,N}^{\frac{1}{2}}}\right)\sqrt{\frac{Nn!}{N!}}\frac{1}{(n-N-1)!}\tilde{x}_N^{\frac{n-N-2}{2}}\\
       &:=K_{n,N}\tilde{x}_N^{\frac{n-N-2}{2}}\le K_{n,N}\left(\frac{\tilde{\mu}_{n,N}}{e}\right)^{\frac{n-N-2}{2}}:=\tilde{K}_{n,N}.
\end{split}.
\end{equation}
By Stirling's formula: $\log(m!)=\frac{1}{2}\log 2\pi+(m+\frac{1}{2})\log m-m+O\left(\frac{1}{m}\right)$, the logarithm of $\tilde{K}_{n,N}$ has the following large $N$ expansion:
\begin{equation}\label{logK'1}
    \begin{split}
       \log\tilde{K}_{n,N}&=\frac{\log (n)!+\log N
        -\log(N!)}{2}-\log((n-N-1)!)\\&~~+\frac{1}{2}\log\frac{\sigma_{n-1,N}}{\sigma_{n,N-1}}+\log\tilde{\sigma}_{n,N}+\frac{n-N-2}{2}\left(\log\tilde{\mu}_{n,N}\right)-\frac{n-N-2}{2}-\frac{1}{2}\log2\\
        &=\frac{n}{2}\log n-\frac{N}{2}\log N-\left(n-N-\frac{1}{2}\right)\log\left(n-N-1\right)\\
        &~~+\frac{1}{4}\left(\log n+\log N+2\log\sigma_{n-1,N}-2\log\sigma_{n,N}+4\log\tilde{\sigma}_{n,N}\right)\\&~~+\frac{n-N-2}{2}\log\tilde{\mu}_{n,N}-\log\sqrt{4\pi}+O\left(N^{-1}\right).
    \end{split}
\end{equation} 
From the order in \eqref{restriction of recaling}, it follows that $\tilde{\mu}_{n,N}=\mu_{n,N}+O(N^{-1})$, $\frac{\sigma_{n,N}}{\sigma_{n-1,N}}=1+O\left(N^{-1}\right)$, $\frac{\sigma_{n,N}}{\tilde\sigma_{n,N}}=1+O\left(N^{-1}\right)$ and hence,
\begin{equation}\label{logK'2}
\begin{split}
        &\frac{n-N-2}{2}\log\tilde{\mu}_{n,N}+\frac{1}{4}\left(\log n+\log N+2\log\sigma_{n-1,N}-2\log\sigma_{n,N}+4\log\tilde{\sigma}_{n,N}\right)\\
        &=\frac{n-N-2}{2}\log \mu_{n,N}+\frac{n-N-2}{2}\log\frac{\tilde{\mu}_{n,N}}{\mu_{n,N}}+\frac{1}{4}\log n+\frac{1}{4}\log N+\frac{1}{2}\log\sigma_{n,N}+O(N^{-1})\\
        &=\frac{n-N}{2}\log n+\frac{n-N}{2}\log F(n,N)+\frac{5}{12}\log n+\frac{1}{4}\log N+O(1)
\end{split}
\end{equation}
where 
\begin{equation*}
    \begin{array}{cc}
       F(n,N):=\left(\sqrt{1+\frac{1}{2n}}-\sqrt{\frac{N}{n}+\frac{1}{2n}}\right)^2
    \end{array}
\end{equation*}
is the function such that $nF(n,N)=\mu_{n,N}$. Plugging \eqref{logK'2} back into \eqref{logK'1} further leads to:
\begin{equation*}
    \begin{split}
     \log\tilde{K}_{n,N}&=-\frac{N}{2}\log\frac{N}{n}+\frac{n-N}{2}\left(\log F(n,N)-2\log\left(1-\frac{N}{n}-\frac{1}{n}\right)\right)+\Theta\left(\log(Nn)\right)=\frac{n}{2}\tilde{F}(n,N).
    \end{split}
\end{equation*}
The limit:
\begin{equation}\label{lim4K}
\begin{split}
  \lim_{N\rightarrow\infty}\tilde{F}(n,N)&=-\gamma\log\gamma+(1-\gamma)\left(-2\log(1-\gamma)+2\log(1-\sqrt\gamma)\right)\\
  &=-\gamma\log\gamma+2(\gamma-1)\log(1+\sqrt\gamma)
  \end{split}
\end{equation}
is immediate by the assumption that $\frac{N}{n}\rightarrow\gamma$.
Now focus on the function $h(t):=t^2\log t+(1-t^2)\log(1+t)$ with $t\in(0,1)$ for a moment. It is easy to see $h$ is a differentiable function with $h(0_+)=h(1_-)=0$. Furthermore, directly taking the derivative gives:
\begin{equation*}
    h'(t)=2t\log t-2t\log(1+t)+1=1-2t\log\left(1+\frac{1}{t}\right).
\end{equation*}
Using the fact that $\left(1+\frac{1}{t}\right)^t$ is a strictly increasing function on $\mathbb{R}^+$, we see $h'$ is a strictly decreasing function on $(0,1)$ with $h'(0_-)=1>0$ and $h'(1_+)=1-\log 4<0$. The intermediate value theorem yields the existence of a unique $t_0\in(0,1)$ such that $h'(t_0)=0$ and thereby, $h$ is strictly increasing on $(0,t_0)$ and strictly decreasing on $(t_0,1)$. As a result, $h(t)>0$ for any $t\in(0,1)$. Since $-2h(\sqrt\gamma)$ equals the limit of $\tilde{F}(n,N)$ in \eqref{lim4K}, for $N$ large enough, there exists a positive constant $C(\gamma)$ determined by $\gamma$ such that $\tilde{F}(n,N)\le -C(\gamma)$ and hence, $r\log N+\log\tilde{K}_{n,N}\le-C(\gamma,r)n$ for any fixed $r$. Plugging it back into \eqref{varsmall} when $N$ is large enough, we finally get
\begin{equation*}
\begin{split}
       N^r|\xi_{\tau}(t)|&=N^r|\xi_{\tau}(t)e^{\frac{t}{2}}e^{-\frac{t}{2}}|\le e^{-C(\gamma,r)n}\left(\sup_{t\in\left[0,\frac{\tilde{\mu}_{n,N}}{\tilde{\sigma}_{n,N}}\right)}e^{\frac{t}{2}}\right)e^{-\frac{t}{2}}\\
       &=\exp\left(-C(\gamma,r)n+\frac{\tilde{\mu}_{n,N}}{2\tilde{\sigma}_{n,N}}\right)e^{-\frac{t}{2}}\le C(\gamma,r)e^{-\frac{t}{2}}
\end{split}
\end{equation*}
where we used the fact that $\frac{\tilde{\mu}_{n,N}}{\tilde{\sigma}_{n,N}}=O\left(n^\frac{2}{3}\right)$ in the second line. The argument for $H_{\tau}$ follows from the similar argument.\qed

\subsection{Order of Errors}
In this part, we will analyze the error term $\epsilon_2$ in our L--G approximation. Following the approach in Olver's treatment of the largest turning point $z_2$ (see \cite[Chapter $11$]{olver1997asymptotics}), we now turn our attention to bounding $\epsilon_2$ around the turning point $z_1$. To this end, we make use of the following lemma, which is adapted from Theorem $6.10.2$ in \cite{olver1997asymptotics}.

\begin{lemma}\label{intBound}
    For any interval $I=(a,b)\subset\mathbb{R}$ where $a$ and $b$ are allowed to be infinite, consider the integral equation:
    \begin{equation}\label{I-equa1}
h(x)=\int_a^xK(x,t)\phi(t)\left[J(t)+h(t)\right]dt.
    \end{equation}
Assume the following conditions hold:
\begin{itemize}
    \item The functions $\phi$ and $J$ are continuous on $I$.
    \item The kernel $K(x,t)$ and its first two partial $x$ derivatives are all continuous for any $(x,t)\in I^2$. Additionally, $K(x,x)=0$ for any $x\in I$.
    \item There exist positive continuous functions $P_0,~P_1$ and nonnegative continuous $Q$ on the real line such that for any $(x,t)\in I^2$,
    \begin{equation}\label{QPres}
        |K(x,t)|\le P_0(x)Q(t),~~|\partial_1K(x,t)|\le P_1(x)Q(t).
    \end{equation}
    \item For any $x\in I$, $\Phi(x):=\int_a^x|\phi(t)|dt$ is convergent. Moreover, 
    \begin{equation}\label{Cres}
        \kappa:=\sup_{x\in I}|Q(x)J(x)|~~and~~\kappa_0:=\sup_{x\in I}|P_0(x)Q(x)|
    \end{equation} are all finite.
    \item If $a=-\infty$, then additionally there exist $C\ge0$ and $R>0$ such that $|Q(x)|\le C$ for any $x\le-R$.
\end{itemize}
Then there is a unique solution $h$ of \eqref{I-equa1} such that
\begin{equation}\label{expB4h}
    \frac{|h(x)|}{P_0(x)},~\frac{|h'(x)|}{P_1(x)}\le\frac{\kappa}{\kappa_0}\left(e^{\kappa_0\Phi(x)}-1\right).
\end{equation}
\end{lemma}

\begin{proposition}\label{ex1}
    Let $\textbf{M}$ and $\textbf{E}$ be the controlling functions introduced above. Then we have:
    \begin{equation*}
        \left|r_N\left(\frac{\kappa_N}{\sigma_{n,N}^3}\right)^{\frac{1}{6}}\frac{\mathbf{M}\left(\kappa_N^\frac{2}{3}\zeta\right)}{\tilde{f}^\frac{1}{4}(z)\mathbf{E}\left(\kappa_N^\frac{2}{3}\zeta\right)}\left(\frac{\mu_{n,N}}{x_N(s)}\right)\right|\le C(s_0)e^{-s}
    \end{equation*}
    for any $s\in\left[s_0,\frac{2\mu_{n,N}}{3\sigma_{n,N}}\right]$ when $N$ is large enough where $C(s_0)$ is a constant determined by $s_0$.
\end{proposition}
\begin{proof}
 From the explicit expression of parameters independent of $s$, we see
\begin{equation*}
     \left|r_N\left(\frac{\kappa_N}{\sigma_{n,N}^3}\right)^{\frac{1}{6}}\left(\frac{\mu_{n,N}}{x_N(s)}\right)\right|\le C
\end{equation*}
for some constant $C$. Actually, the boundedness of the first two terms is straightforward by their respective order. For the third term, as $1-\frac{\sigma_{n,N}}{\mu_{n,N}}s\ge\frac{1}{3}$ for any $s$ in the required regime, we have $x_N(s)=\mu_{n,N}-\sigma_{n,N}s\ge\frac{\mu_{n,N}}{3}$ and hence, $\left|\frac{\mu_{n,N}}{x_N(s)}\right|\le3$ for any $s$ in our setting.
As the controlling functions satisfy $\mathbf{E}(x)\ge c_0e^{\frac{2}{3}x^\frac{3}{2}}$ and $\mathbf{M}(x)\le c_1x^{-\frac{1}{4}}$ for any $x\ge 0$ and absolute constants $c_0$, $c_1$ (see \cite[Chapter $11$]{olver1997asymptotics}). Let $s_1$ be as in Proposition \ref{B4z}. For $s\in\left[s_1,\frac{2\mu_{n,N}}{3\sigma_{n,N}}\right]$, we have $\frac{1}{\mathbf{E}\left(\kappa_N^\frac{2}{3}\zeta\right)}\le c_0^{-1}e^{-\frac{2}{3}\kappa_N\zeta^{\frac{3}{2}}}\le Ce^{-s}$ according to the same proposition. Furthermore, as in the proof of this proposition again, the fact that $f(z)\ge\frac{z_2-z_1}{4z_1^2}\frac{\sigma_{n,N}}{\kappa_N}s$ leads to $\kappa_N^{\frac{2}{3}}f(z)\ge C'$ for some constant $C'$ when $N$ is large. These bounds above derive:
\begin{equation*}
    \tilde{f}^{-\frac{1}{4}}(z)\mathbf{M}(\kappa_N^{\frac{2}{3}}\zeta)\le c_1\frac{\zeta^\frac{1}{4}}{f^\frac{1}{4}}\kappa_N^{-\frac{1}{6}}\zeta^{-\frac{1}{4}}=c_1\left(\kappa_N^\frac{2}{3}f\right)^{-\frac{1}{4}}\le C''
\end{equation*}
for some positive constant $C''$ and hence, the required inequality is true for $s\in\left[s_1,\frac{2\mu_{n,N}}{3\sigma_{n,N}}\right]$.

For $s\in[s_0,s_1]$, we have $z=z_1+\Theta\left(N^{-\frac{2}{3}}\right)$. The limit $\lim\limits_{z\rightarrow z_1}\tilde{f}(z)=\dot{\zeta}^2(z_1)$ yields the fact that $\tilde{f}^{-\frac{1}{4}}(z)\le \frac{1}{2}\dot{\zeta}^{-\frac{1}{2}}({z_1})\le C$ for some constant $C$ and large $N$. Since $\mathbf{ M}\le1$ and $\mathbf{ E}\ge 1$, the left side of the required inequality is bounded by some positive constant $C$ on $[s_0,s_1]$ when $N$ is large. Define $L(s_0):=\sup\limits_{t\in[s_0,s_1]}e^{t}<\infty$, it follows that
\begin{equation*}
      \left|r_N\left(\frac{\kappa_N}{\sigma_{n,N}^3}\right)^{\frac{1}{6}}\frac{\mathbf{M}\left(\kappa_N^\frac{2}{3}\zeta\right)}{\tilde{f}^\frac{1}{4}(z)\mathbf{E}\left(\kappa_N^\frac{2}{3}\zeta\right)}\left(\frac{\mu_{n,N}}{x_N(s)}\right)\right|\le CL(s_0)
\end{equation*}
and hence, the required inequality is also true on $[s_0,s_1]$.\end{proof}

\begin{proposition}\label{Psi}
     Let $\Psi(\zeta)$ be the perturbation introduced in \eqref{W-equa2'} where $\zeta$ is the L--G transform around $z_1$ as in \eqref{LG1}. Then the following large $\zeta$ (small $z$) expansions hold:
     \begin{equation}\label{order of pert}
     \tilde{f}\sim \left(\frac{\kappa_N}{6a}\right)^\frac{2}{3}\frac{z_1z_2}{z^2\log^\frac{2}{3}\left(\frac{1}{z}\right)},~~\tilde{f'}\sim \left(\frac{\kappa_N}{12\sqrt 2a}\right)^\frac{2}{3}\frac{z_1z_2}{z^3\log^\frac{2}{3}\left(\frac{1}{z}\right)}~and~\Psi(\zeta)\sim\frac{5}{16\zeta^2}
     \end{equation}
     where the prime denotes the derivative with respect to $z$.
\end{proposition}
\begin{proof}
    From the expression in \eqref{LG1}, we see $\zeta\rightarrow\infty$ when $z\rightarrow0$ and thus, the large $\zeta$ expansion is equivalent to the small $z$ expansion for all functions above. In the following argument, we mainly focus on the small $z$ expansion for $\tilde{f}$ and large $\zeta$ expansion for $\Psi$ due to the convenience of notations. Since $\zeta^\frac{3}{2}\sim\frac{3a}{4\kappa_N}\log\frac{1}{z}$ according to \eqref{ex4zeta1}, it is clear that 
    \begin{equation}\label{B4tildef}
        \tilde{f}(z)=\frac{f}{\zeta}(z)=\frac{z^2-4z+z_1z_2}{4z^2\zeta(z)}\sim\left(\frac{4\kappa_N}{3a}\right)^\frac{2}{3}\frac{z_1z_2}{4}\frac{1}{z^2\log^\frac{2}{3}\left(\frac{1}{z}\right)}.
    \end{equation}
Moreover, since $\zeta'=-\frac{f^\frac{1}{2}}{\zeta^\frac{1}{2}}$ by \eqref{LG1},
\begin{equation}\label{B4tildef'}
\begin{split}
       \tilde{f}'(z)&=\frac{f'\zeta-\zeta'f}{\zeta^2}(z)=\frac{f'\zeta^\frac{3}{2}+f^\frac{3}{2}}{\zeta^\frac{5}{2}}(z)=\frac{1}{z^2\zeta}-\frac{z_1z_2}{8z^3\zeta}+\frac{\left(z^2-4z+z_1z_2\right)^\frac{3}{2}}{8z^3\zeta^\frac{5}{2}}\\
       &\sim-\left(\frac{4\kappa_N}{3a}\right)^\frac{2}{3}\frac{z_1z_2}{8}\frac{1}{z^3\log^\frac{2}{3}\left(\frac{1}{z}\right)}.
\end{split}
\end{equation}
Next, let's express $\Psi$ in \eqref{W-equa2'} in terms of $\zeta$ and derivatives of $f$ to capture the main order when $\zeta$ is large. Noticing that
\begin{equation*}
    \frac{d^2}{dz^2}\left(\tilde{f}^{-\frac{1}{4}}\right)=\frac{d^2}{dz^2}\left(\frac{1}{\sqrt{-\zeta'}}\right)=\frac{d}{dz}\left(\frac{\zeta''}{2\left(-\zeta'\right)^\frac{3}{2}}\right)=\frac{\zeta'''}{2\left(-\zeta'\right)^\frac{3}{2}}+\frac{3(\zeta'')^2}{4\left(-\zeta'\right)^\frac{5}{2}},
\end{equation*}
we see
\begin{equation*}
    -\tilde{f}^{-\frac{3}{4}}\frac{d^2}{dz^2}\left(\tilde{f}^{-\frac{1}{4}}\right)+\frac{g}{\tilde{f}}=\frac{\zeta'''}{2(\zeta')^3}-\frac{3(\zeta'')^2}{4(\zeta')^4}+\frac{g}{\tilde{f}}=\Psi(\zeta).
\end{equation*}
Again, as $\zeta'=-\frac{f^\frac{1}{2}}{\zeta^\frac{1}{2}}$, by \eqref{B4tildef'}, it follows that
\begin{equation*}
    \begin{split}
        -\tilde{f}^{-\frac{3}{4}}\frac{d^2}{dz^2}\left(\tilde{f}^{-\frac{1}{4}}\right)&=-\tilde{f}^{-\frac{3}{4}}\frac{d}{dz}\left(-\frac{1}{4}\tilde{f}^{-\frac{5}{4}}\tilde{f}'\right)=-\frac{5}{16}\tilde{f}^{-3}\left(\tilde{f}'\right)^2+\frac{1}{4}\tilde{f}^{-2}\tilde{f}''\\
        &=\frac{5}{16\zeta^2}-\frac{5\zeta(f')^2}{16f^3}+\frac{f''\zeta}{4f^2}=\frac{5}{16\zeta^2}+\frac{\zeta}{16f^3}\left(4ff''-5(f')^2\right)
    \end{split}
\end{equation*}
and hence,
\begin{equation}\label{explicit4Psi}
    \begin{split}
        \Psi(\zeta)&=\frac{5}{16\zeta^2}+\frac{4\zeta z^6}{(z^2-4z+z_1z_2)^3}\left(-\frac{2}{z^3}+\frac{6+3z_1z_2}{2z^4}-\frac{3z_1z_2}{z^5}+\frac{(z_1z_2)^2}{4z^6}\right)+\frac{g\zeta}{f}\\
        &=\frac{5}{16\zeta^2}-\frac{\zeta}{(z^2-4z+z_1z_2)^3}\left(z^4+(4-z_1z_2)z^2+4z_1z_2z\right)\sim\frac{5}{16\zeta^2}
    \end{split}
\end{equation}
since $\zeta^\frac{3}{2}\sim\frac{3a}{4\kappa_N}\log\frac{1}{z}$ when $z\rightarrow0$.
\end{proof}

With the estimations in Proposition \ref{Psi} and useful auxiliary functions introduced in Proposition \ref{aux4AiandBi}, applying Lemma \ref{intBound}, we have the following bound regarding $\epsilon_2$.

\begin{proposition}\label{epsilon_2}
    Let $\epsilon_2(\kappa_N,z)$ be the function of $\zeta$ in \eqref{w2}, then
      \begin{equation}\label{B4e_2}
        |\epsilon_2(\kappa_N,z)|\le \frac{\mathbf{M}(\kappa_N^{2/3}\zeta)}{\mathbf{E}(\kappa_N^{2/3}\zeta)}\left[\exp\left(\frac{\lambda_0 \mathcal{V}(\zeta)}{\kappa_N}\right)-1\right].
    \end{equation}
The function $\mathcal{V}$ is defined as $\mathcal{V}(\zeta):=\int_\zeta^\infty|\Psi(t)|t|^{-\frac{1}{2}}dt$ and the constant $\lambda_0:=\sup\limits_{x\in \mathbb{R}}\left\{\pi|x|^\frac{1}{2}\mathbf{M}^2(x)\right\}<\infty$.
\end{proposition}
\begin{proof}
    Recalling that $W_2=Ai\left(\kappa_N^\frac{2}{3}\zeta\right)+\epsilon_2(\zeta)$ is a solution of \eqref{W-equa2}, then $\epsilon_2$ satisfies
    \begin{equation}\label{W-epsilon1}
        \frac{d^2\epsilon_2(\zeta)}{d\zeta^2}-\kappa_N^2\zeta\epsilon_2(\zeta)=\Psi(\zeta)\left(Ai\left(\kappa_N^\frac{2}{3}\zeta\right)+\epsilon_2(\zeta)\right)
    \end{equation}
    since $Ai\left(\kappa_N^\frac{2}{3}\zeta\right)$ is a solution of $\frac{d^2W(\zeta)}{d\zeta^2}=\kappa_N^2\zeta W(\zeta)$.
    
    To give a concrete expression of $\epsilon_2$, let's recall some basic results on second order differential equations. Consider a second order linear differential equation with the standard form:
    \begin{equation}\label{second order}
        y''(x)+p(x)y'(x)+q(x)y(x)=g(x).
    \end{equation}
    Suppose $p,~q$ and $g$ are all continuous functions. Then the solution $y$ of \eqref{second order}, if it exists, is unique given some initial conditions: $y(x_0)=y_0$ and $y'(x_0)=y_1$. Now let's check the initial condition of $\epsilon_2$. In fact, according to \eqref{w2}, $\tilde{f}^\frac{1}{4}(z)w_N(\kappa_Nz)=c_NAi\left(\kappa_N^\frac{2}{3}\zeta\right)+c_N\epsilon_2(\zeta)$. Since the leading term of the left side is of order $\frac{z^\frac{a+1}{2}}{z^\frac{1}{2}\log^\frac{1}{6}\log\left(\frac{1}{z}\right)}$ from \eqref{B4tildef} for small $z$, which vanishes as $z\rightarrow0$, we see $\epsilon_2(\zeta)$ also converges to $0$ because $Ai\left(\kappa_N^\frac{2}{3}\zeta\right)\rightarrow0$ when $z\rightarrow0$. Further condition can be obtained by taking the derivative:
    \begin{equation*}
        \frac{1}{4\tilde{f}^\frac{3}{4}}\tilde{f}'w_N(\kappa_Nz)+\tilde{f}^\frac{1}{4}w_N'(\kappa_Nz)\kappa_N-c_NAi'\left(\kappa_N^\frac{2}{3}\zeta\right)\kappa_N^\frac{2}{3}\zeta'=c_N\epsilon_2'(\zeta)\zeta'
    \end{equation*}
    and hence,
        \begin{equation*}
        \frac{1}{4\tilde{f}^\frac{3}{4}}\tilde{f}'w_N(\kappa_Nz)\frac{1}{\zeta'}+\tilde{f}^\frac{1}{4}w_N'(\kappa_Nz)\frac{\kappa_N}{\zeta'}-c_NAi'\left(\kappa_N^\frac{2}{3}\zeta\right)\kappa_N^\frac{2}{3}=c_N\epsilon_2'(\zeta).
    \end{equation*}
    Similarly, according to \eqref{B4tildef}, \eqref{B4tildef'}, the leading orders of the first two terms on the left side are $z^\frac{a}{2}$ and $z^\frac{a}{2}\log^{-\frac{5}{6}}\left(\frac{1}{z}\right)$ respectively since $\zeta'\sim\frac{1}{z\log\frac{1}{z}}$. As a result, they both vanish as $z\rightarrow0$. Combining the fact that $Ai'(x)$ vanishes when $x\rightarrow\infty$, it follows that $\epsilon'_2(\zeta)\rightarrow0$ as $z\rightarrow0$. Consequently, $\epsilon_2$ is the unique solution of \eqref{W-epsilon1} satisfying the initial condition: $\epsilon_2(\zeta)|_{z=0}=\epsilon_2'(\zeta)|_{z=0}=0$.
    
    Assume $y_1$ and $y_2$ are two linearly independent homogeneous solutions of \eqref{second order} (solutions with respect to $g=0$). Then by the Wronskian method, a particular solution of \eqref{second order} is given as:
    \begin{equation}\label{particular solution}
        y_p(x)=y_2(x)\int\frac{y_1(x)g(x)}{\mathcal{W}_{y_1,y_2}(x)}dx-y_1(x)\int\frac{y_2(x)g(x)}{\mathcal{W}_{y_1,y_2}(x)}dx
    \end{equation}
    where $\mathcal{W}_{y_1,y_2}(x):=y_1(x)y'_2(x)-y_2(x)y_1'(x)$ is the Wronskian of $y_1$ and $y_2$. Furthermore, since $y_1$ and $y_2$ are solutions of the homogeneous equation, it is easy to verify $\mathcal{W}'_{y_1,y_2}(x)+p(x)\mathcal{W}_{y_1,y_2}(x)=0$ and hence, $\mathcal{W}_{y_1,y_2}(x)=C\exp\left(-\int p(x)dx\right)$ for $C\in\mathbb{R}$. Regarding $\eqref{W-epsilon1}$ as a standard form with $p(\zeta)=0$, $q(\zeta)=\kappa_N^2\zeta$ and $g(\zeta)=\Psi(\zeta)\left(Ai\left(\kappa_N^\frac{2}{3}\zeta\right)+\epsilon_2(\zeta)\right)$, then clearly $Ai\left(\kappa_N^\frac{2}{3}\zeta\right)$ and $Bi\left(\kappa_N^\frac{2}{3}\zeta\right)$ are two linearly independent homogeneous solutions from our previous reasoning. Moreover, since $p(\zeta)=0$ in \eqref{W-epsilon1}, the Wronskian $\mathcal{W}_{Ai,Bi}(x)$ must be a constant function. According to special values of $Ai$ and $Bi$ at $0$ (see \cite[Equation $11.1.03\&11.1.11$]{olver1997asymptotics}),
    \begin{equation}\label{W4AiBi}
              \mathcal{W}_{Ai,Bi}(x)=Ai(0)Bi'(0)-Ai'(0)Bi(0)=\frac{2}{\sqrt3}\frac{1}{\Gamma\left(\frac{2}{3}\right)\Gamma\left(\frac{1}{3}\right)}=\frac{2}{\sqrt3}\frac{\sin\frac{\pi}{3}}{\pi}=\frac{1}{\pi}
    \end{equation}
    where Euler's reflection formula: $\Gamma(z)\Gamma(1-z)=\frac{\pi}{\sin(z\pi)}$ for $z\notin\mathbb{Z}$ has been used. According to \eqref{particular solution}, as a particular solution of \eqref{W-epsilon1} vanishes at $\infty$, $\epsilon_2$ can be taken as:
    \begin{equation}\label{W-epsilon2}
        \begin{split}
            \epsilon_2(\zeta)=\int^\infty_\zeta K(\zeta,t)\tilde{\Psi}(t) \left(Ai\left(\kappa_N^\frac{2}{3}t\right)+\epsilon_2(t)\right)dt
        \end{split}
    \end{equation}
    where
    \begin{equation*}
    \begin{split}
           &K(\zeta,t)=\pi\kappa_N^{-\frac{2}{3}}|t|^\frac{1}{2} \left(Ai\left(\kappa_N^\frac{2}{3}\zeta\right)Bi\left(\kappa_N^\frac{2}{3} t\right)-Bi\left(\kappa_N^\frac{2}{3}\zeta\right)Ai\left(\kappa_N^\frac{2}{3} t\right)\right):=\pi\kappa_N^{-\frac{2}{3}}|t|^\frac{1}{2}\tilde{K}(\zeta,t)
    \end{split}
    \end{equation*}
    and $\tilde{\Psi}(t)=\Psi(t)|t|^{-\frac{1}{2}}$. The term $|t|^\frac{1}{2}$ has been pulled out to guarantee the finiteness of controlling constants as in Lemma \ref{intBound}. Using the phase function in \eqref{MandEzero}, for any $t\ge\zeta$, let $u:=\kappa_N^\frac{2}{3}\zeta$ and $v:=\kappa_N^\frac{2}{3}t$, then
    \begin{equation}\label{B4K1}
        \begin{split}
            |\tilde{K}(\zeta,t)|&=\mathbf{M}\left(u\right)\mathbf{M}\left(v\right)\frac{\mathbf{E}\left(v\right)}{\mathbf{E}\left(u\right)}\left|\sin\theta_u\cos\theta_v-\frac{\mathbf{E}^2\left(u\right)}{\mathbf{E}^2\left(v\right)}\sin\theta_v\cos_u\right|\\
&\le\left(\mathbf{M}\left(u\right)\mathbf{M}\left(v\right)\frac{\mathbf{E}\left(v\right)}{\mathbf{E}\left(u\right)}\right)\left|\frac{\mathbf{E}^2(u)}{\mathbf{E}^2(v)}\sin\left(\theta_u-\theta_v\right)+\left(1-\frac{\mathbf{E}^2(u)}{\mathbf{E}^2(v)}\right)\sin\theta_u\cos\theta_v\right|\\
&\le\mathbf{M}\left(u\right)\mathbf{M}\left(v\right)\frac{\mathbf{E}\left(v\right)}{\mathbf{E}\left(u\right)}
        \end{split}
    \end{equation}
    since $\frac{\mathbf{E}^2(u)}{\mathbf{E}^2(v)}\in(0,1]$ by the monotonicity of $\mathbf{E}$ on $\mathbb{R}^+$. Similarly, the phase function in \eqref{NandEfirst} yields:
    \begin{equation}\label{B4K2}
        \begin{split}
            |\partial_1\tilde{K}(\zeta,t)|&=\kappa_N^\frac{2}{3}\left|Ai'(u)Bi(v)-Ai(v)Bi'(u)\right|\\
            &=\kappa_N^\frac{2}{3}\mathbf{N}(u)\mathbf{M}(v)\frac{\mathbf{E}(v)}{\mathbf{E}(u)}\left|\sin\omega_u\cos\theta_v-\frac{\mathbf{E}^2(u)}{\mathbf{E}^2(v)}\sin\theta_v\cos\omega_u\right|\\&\le\kappa_N^\frac{2}{3}\mathbf{N}(u)\mathbf{M}(v)\frac{\mathbf{E}(v)}{\mathbf{E}(u)}.
        \end{split}
    \end{equation}
    Define $Q(t):=\pi\kappa_N^{-\frac{2}{3}}|t|^\frac{1}{2}\mathbf{M}\left(\kappa_N^\frac{2}{3}t\right)\mathbf{E}\left(\kappa_N^\frac{2}{3}t\right)$, $P_0(\zeta):=\frac{\mathbf{M}\left(\kappa_N^\frac{2}{3}\zeta\right)}{\mathbf{E}\left(\kappa_N^\frac{2}{3}\zeta\right)}$ and $P_1(\zeta):=\frac{\mathbf{N}\left(\kappa_N^\frac{2}{3}\zeta\right)}{\mathbf{E}\left(\kappa_N^\frac{2}{3}\zeta\right)}$, it follows that
    \begin{equation*}
        |K(\zeta,t)|\le P_0(\zeta)Q(t)~~and~~|\partial_1K(\zeta,t)|\le P_1(\zeta)Q(t)
    \end{equation*}
    from \eqref{B4K1} and \eqref{B4K2}. Moreover, according to \cite[Equation $(11.2.15)$]{olver1997asymptotics}, $\lambda_0:=\sup\limits_{x\in\mathbb{R}}\left\{\pi\left|x\right|^\frac{1}{2}\mathbf{M}^2(x)\right\}$ is a finite number which is slightly greater than $1$ and $\lambda_1:=\sup\limits_{x\in\mathbb{R}}\left\{\pi\mathbf{E}(x)\mathbf{M}(x)\left|x\right|^\frac{1}{2}|Ai(x)|\right\}=1$. Now going back to the controlling constants as in \eqref{Cres}, since the function $J(t)$ is taken as $Ai\left(\kappa_N^\frac{2}{3}t\right)$ in our setting,
    \begin{equation*}
 \begin{array}{cc}
\kappa:=\sup\limits_{t\in\mathbb{R}}|Q(t)J(t)|=\sup\limits_{t\in\mathbb{R}}\left|\pi\kappa_N^{-\frac{2}{3}}|t|^\frac{1}{2}\mathbf{M}\left(\kappa_N^\frac{2}{3}t\right)\mathbf{E}\left(\kappa_N^\frac{2}{3}t\right)Ai\left(\kappa_N^\frac{2}{3}t\right)\right|=\frac{1}{\kappa_N},\\
\\
\kappa_0:=\sup\limits_{t\in\mathbb{R}}|Q(t)P_0(t)|=\sup\limits_{t\in\mathbb{R}}\left|\pi\kappa_N^{-\frac{2}{3}}|t|^\frac{1}{2}\mathbf{M}^2\left(\kappa_N^\frac{2}{3}t\right)\right|=\frac{\lambda_0}{\kappa_N}.
\end{array}
\end{equation*}
The convergence of $\Phi(z)$, or equally, $\mathcal{V}(\zeta)$, is straightforward by the expansion of $\Psi(\zeta)$ in Proposition \ref{Psi}, which will be discussed later. Consequently, Lemma \ref{intBound} applies and hence,
\begin{equation*}
    |\epsilon_2(\kappa_N,z)|\le P_0(\zeta)\frac{\kappa}{\kappa_0}\left(e^{\kappa_0\Phi(z)}-1\right)\le\frac{\mathbf{M}(\kappa_N^{2/3}\zeta)}{\mathbf{E}(\kappa_N^{2/3}\zeta)}\left[\exp\left(\frac{\lambda_0 \mathcal{V}(\zeta)}{\kappa_N}\right)-1\right]
\end{equation*}
since $\lambda_0>1$.
\end{proof}

\begin{proposition}\label{ex2}
    Let $\epsilon_2(\kappa_N,z)$ be the error term and $\mathcal{V}(\omega_N,\zeta)$ be the controlling function introduced before. Then for any $\delta\in(0,1)$, we have $\mathcal{V}(\omega,\zeta)$ is bounded for any $\omega\in[\delta,2-\delta]$ and $s\in\left[s_0,\frac{2\mu_{n,N}}{3\sigma_{n,N}}\right]$ when $N$ is large. Hence, the error is controlled by
    \begin{equation}\label{B4epsilon}
        |\epsilon_2(\kappa_N,z)|\le C(\delta)\frac{\mathbf{M}\left(\kappa_N^\frac{2}{3}\zeta\right)}{\kappa_N\mathbf{E}\left(\kappa_N^\frac{2}{3}\zeta\right)}
    \end{equation}
    for some constant $C(\delta)$ determined by $\delta$ when $N$ is large under our assumption.
\end{proposition}
\begin{proof}
    Adapting the method in \cite{el2006rate} of the discussion around $z_2$, the bound around $z_1$ can be obtained similarly. Since $z=\frac{x_N(s)}{\kappa_N}$, we see $0<z<2$ for any $s\in\left[s_0,\frac{2\mu_{n,N}}{3\sigma_{n,N}}\right]$ when $N$ is large. Hence, \cite[Lemma $11.3.1$]{olver1997asymptotics} applies if we replace $(a,b)$ by $(0,2]$ and $x_0$ by $z_1$. Recalling $\zeta\rightarrow\infty$ as $z\rightarrow0$, we first split $(0,2]$ into two pieces for the convenience of argument. In fact, according to Proposition \ref{Psi}, $\Psi(\zeta)\sim\frac{5}{16\zeta^2}$ when $\zeta$ is large and hence, there exists $\zeta_0$ large enough such that $\Psi(\zeta)\le\frac{1}{\zeta^2}$ for any $\zeta>\zeta_0$. Let $z_0$ be the corresponding value such that $\zeta(z_0)=\zeta_0$, then split $(0,2]$ into $(0,z_0)\cup[z_0,2]$. For the second piece, recalling that
    \begin{equation*}
        \frac{2}{3}\zeta^\frac{3}{2}(z)=\int_{z}^{z_1}f^{1/2}(\omega_N,t)dt
    \end{equation*}
    for $z\le z_1$ and any fixed $\omega_N$, we have
    \begin{equation*}
        \tilde{f}(z)=p^2(\omega_N,z)\left(\frac{3}{2}q(\omega_N,z)\right)^{-\frac{2}{3}}
    \end{equation*}
    where
    \begin{equation*}
        p(\omega_N,z)=\frac{(z_2-z)^{1/2}}{2z},~and~~q(\omega_N,z)=\frac{1}{(z_1-z)^{3/2}}\int_{z}^{z_1}(z_1-t)^{1/2}p(\omega_N,t)dt.
    \end{equation*}
    As indicated by Olver in \cite{olver1997asymptotics}, $q(\omega_N,z)\rightarrow\frac{2}{3}p(\omega_N,z_1)>0$ as $z\rightarrow z_1$, we see both $p(\omega,z)$ and $q(\omega,z)$ are bounded away from $0$ on $[\delta,2-\delta]\times[z_0,2]$. This fact makes both $\tilde{f}(\omega,z)$ and $\tilde{f}^{-1}(\omega,z)$ well-behaved functions on $[\delta,2-\delta]\times[z_0,2]$ and hence,
    \begin{equation*}
        \Psi(\zeta)=\tilde{f}^{-\frac{1}{4}}\frac{d^2\left(\tilde{f}^{\frac{1}{4}}\right)}{d\zeta^2}+\frac{g}{\tilde{f}}
    \end{equation*}
    a continuous function on the same domain. According the the explicit L--G transform in \eqref{LG1}, it is easy to see $\zeta$ is a nonincreasing function with respect to $z$. So for any $\omega\in[\delta,2-\delta]$, we conclude $\zeta_2(\omega)\le\zeta(\omega):=\zeta(\omega,z)\le\zeta_0(\omega)$ where $\zeta_2$ equals the value of $\zeta$ corresponding to $z=2$.
Combining the explicit expression of $\mathcal{V}$ defined in Proposition \ref{epsilon_2}, for any $\omega\in[\delta,2-\delta]$ and $N$ large enough,
    \begin{equation*}
    \begin{split}
\mathcal{V}(\omega,\zeta)&=\int_{\zeta(\omega)}^\infty|y|^{-\frac{1}{2}}|\Psi(\omega,y)|dy\le\int_{\{z\in[z_0,2]\}}|y|^{-\frac{1}{2}}|\Psi(\omega,y)|dy+\int_{\{z\in(0,z_0)\}}|y|^{-\frac{1}{2}}|\Psi(\omega,y)|dy\\
&\le\int_{\zeta_2(\omega)}^0|y|^{-1/2}|\Psi(\omega,y)|dy+\int_0^{\zeta_0(\omega)}|y|^{-1/2}|\Psi(\omega,y)|dy+\int_{\zeta_0(\omega)}^\infty|y|^{-1/2}|\Psi(\omega,y)|dy\\
&\le\int_{I_2}^0|y|^{-1/2}|\Psi(\omega,y)|dy+\int_0^{S_{0}}|y|^{-1/2}|\Psi(\omega,y)|dy+\int_{I_0}^\infty|y|^{-\frac{5}{2}}dy\\
&:= C(\delta,\omega)\le\sup\limits_{\omega\in[\delta,2-\delta]}C(\delta,\omega):=C(\delta)
    \end{split}
    \end{equation*}
    where $I_r:=\inf_{\omega\in[\delta,2-\delta]}\zeta_r(\omega)$, $S_r:=\sup_{\omega\in[\delta,2-\delta]}\zeta_r(\omega)$ and $C(\delta)$ is some constant determined by $\delta$ only. Here we used the fact $z_0$ can be small such that $\zeta_0$ and $I_0$ are both very large and hence, the convergence of the last integral in the fifth line always holds.
Given the uniform boundedness of $\mathcal{V}$ for $\omega\in[\delta,2-\delta]$, \eqref{B4epsilon} directly follows from the bound in Proposition \ref{epsilon_2}.
\end{proof}

\begin{proposition}\label{B4re}
Let $R_{n-1,N}(s)$ and $R_{n,N-1}(s)$ be the integral remainders defined in \eqref{B4D}. For $s\in\left[s_0,N^{\frac{1}{6}}\right]$, we have
\begin{equation*}
    N^{\frac{2}{3}}R_{n-1,N}(s),~N^\frac{2}{3}R_{n,N-1}(s)\le C(s_0,\gamma)e^{-\frac{s}{2}}
\end{equation*}
when $N$ is large for some constant $C(s_0,\gamma)$.
\end{proposition}
\begin{proof}
    Split $\left[s_0,N^{\frac{1}{6}}\right]$ into $[s_0,2s_0']$ and $\left[2s_0',N^{\frac{1}{6}}\right]$ as in Lemma \ref{keyI2}. For any $s\in[s_0,2s_0']$, we see $-2|s_0|\le\kappa_{n-1,N}^\frac{2}{3}\zeta(\tilde{x}_N(s))=s+d_{n-1,N}(s)\le 4s_0'$ when $N$ is large. Therefore, for any $(s,t)\in[s_0,2s_0']\times[0,d_{n-1,N}(s)]$, $|(s+t)Ai(s+t)|\le M(s_0)$ for some constant $M$ determined by $s_0$ and hence,
    \begin{equation}\label{smallR}
        \begin{split}
            R_{n-1,N}(s)&=\left|\int_0^{d_{n-1,N}(s)}(d_{n-1,N}(s)-t)(s+t)Ai(s+t)dt\right|\\
            &\le M(s_0)\left|\int_0^{d_{n-1,N}(s)}(d_{n-1,N}(s)-t)dt\right|\\&\le C(s_0)(d_{n-1,N}(s))^2
        \end{split}
    \end{equation}
    for $N$ large enough. As $d_{n-1,N}(s)=\Theta\left(N^{-\frac{1}{3}}\right)$ according to \eqref{a4D} in this regime, the inequality:
    \begin{equation}\label{smallR2}
        N^{\frac{2}{3}}R_{n-1,N}(s)\le \left(\sup_{s_0\le s\le 2s_0'}C(s_0,\gamma)e^{\frac{s}{2}}\right)e^{-\frac{s}{2}}=C(s_0,\gamma)e^{-\frac{s}{2}}
    \end{equation}
    follows immediately from \eqref{smallR}.

    For $s\in\left[2s_0',N^{\frac{1}{6}}\right]$, the fact that $\min\{s+d_{n-1,N}(s),s\}\ge\frac{s}{2}\ge1$ when is $N$ is large, has already been proven in our previous discussion for $D_{n,N}$. Combining the property that the Airy function is nonincreasing and positive on this interval, for $N$ large, we get
    \begin{equation}\label{largeR}
        \begin{split}
            R_{n-1,N}(s)&=\left|\int_0^{d_{n-1,N}(s)}(d_{n-1,N}(s)-t)(s+t)Ai(s+t)dt\right|\\&\le Ai\left(\frac{s}{2}\right)\left|\int_0^{d_{n-1,N}(s)}(d_{n-1,N}(s)-t)(s+t)dt\right|\\
            &=Ai\left(\frac{s}{2}\right)\frac{\left|d_{n-1,N}^3(s)+2d^2_{n-1,N}(s)s\right|}{6}\le Ai\left(\frac{s}{2}\right)d_{n-1,N}^2(s)\left(|d_{n-1,N}(s)|+2|s|\right)\\
            &\le  C(s_0,\gamma)Ai\left(\frac{s}{2}\right)d_{n-1,N}^2(s)\left(|s|^3+2|s|\right)\le C(s_0,\gamma)|s|^3Ai\left(\frac{s}{2}\right)d_{n-1,N}^2(s)\\
            &\le C(s_0,\gamma)|s|^9Ai\left(\frac{s}{2}\right)N^{-\frac{2}{3}}
        \end{split}
    \end{equation}
    where we used that fact that $d_{n-1,N}(s)=O\left(s^3N^{-\frac{1}{3}}\right)$ from \eqref{a4D} and $s\ge2$. Using the bound \eqref{ExB4Airy} for the Airy function and the fact that $s\ge2$ when $N$ is large, we further get
    \begin{equation}\label{largeR2}
N^\frac{2}{3}R_{n-1,N}\le C(s_0,\gamma)e^{-\frac{s}{2}}.
    \end{equation}
    The estimations \eqref{smallR2} and \eqref{largeR2} together give us the bound for $R_{n-1,N}$. The promised result for $R_{n,N-1}$ can be obtained by the same argument.
\end{proof}

\begin{proposition}\label{error4d}
    Let $R'_{n-1,N}(s)$ and $R'_{n,N-1}(s)$ be the error terms given in \eqref{a4D}, then
    \begin{equation*}
        N^{\frac{2}{3}}|R'_{n-1,N}(s)||Ai'(s)|,~N^{\frac{2}{3}}|R'_{n,N-1}(s)||Ai'(s)|\le C(s_0,\gamma)e^{-\frac{s}{2}}
    \end{equation*}
    for any $s\in\left[s_0,N^\frac{1}{6}\right]$ when $N$ is large.
\end{proposition}
\begin{proof}
    Since $\tilde{\epsilon}_{n-1,N}(s)=\Theta\left(sN^{-\frac{2}{3}}\right)$, there exists a constant $C(s_0,\gamma)$ such that $|\tilde{\epsilon}_{n-1,N}(s)|\le  N^{-\frac{2}{3}}C(s_0,\gamma)\max\{|s|,1\}$ when $N$ is large enough. Consequently,
    \begin{equation*}
        N^{\frac{2}{3}}|R'_{n-1,N}(s)|\le C(s_0,\gamma)\max\{|s|^3,1\}.
    \end{equation*}
    The bound \eqref{ExB4Airy'} for $Ai'(s)$ leads to the finiteness of
        $M(s_0):=\sup_{s_0\le s\le N^\frac{1}{6}} \left|Ai'(s)\right|e^{\frac{s}{2}}\max\left\{|s|^3,1\right\}$
    when $N$ is large. Hence, for large $N$,
    \begin{equation*}
        N^\frac{2}{3}|R'_{n-1,N}(s)Ai'(s)|=N^\frac{2}{3}|R'_{n-1,N}(s)Ai'(s)e^{s/2}|e^{-\frac{s}{2}}\le C(s_0,\gamma)e^{-\frac{s}{2}}.
    \end{equation*}
    Similar discussion applies to $R'_{n,N-1}(s)$.
\end{proof}